\documentclass[11pt]{article}
\title{Markov chains on trees: almost lower and upper directed cases}
\author{Luis Fredes \& Jean Francois Marckert}
\date{September 2023}
\usepackage{xspace,fourier}

\usepackage[utf8]{inputenc}
\usepackage{amsthm}
\usepackage{amsmath}
\usepackage{amssymb,amsfonts, ulem}
\usepackage{hyperref,stmaryrd,caption}
\usepackage{cleveref}
\usepackage{xcolor,todonotes}
\usepackage[left=2.2cm,right=2.2cm,top=2.5cm,foot=2.5cm]{geometry}

\usepackage[english]{babel}
  
\newcommand{\N}{\mathbb{N}}
\newcommand{\R}{\mathbb{R}}
\newcommand{\C}{\mathbb{C}}

\def \Id{{\sf Id}}

\def \U{{\sf U}}

\def \M{{\sf M}}

\def \sut{{\textsf{A\small{llmost\,}\,U\small{pper}\,T\small{riangular}}}}
\def \slt{{\textsf{A\small{llmost\,}\,L\small{ower}\,T\small{riangular}}}}
\def \sut{\includegraphics[width=0.32cm]{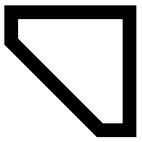}}
\def \slt{\includegraphics[width=0.32cm]{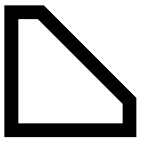}}

\def \Id{{\sf Id}}

\def \U{{\sf U}}

\def \M{{\sf M}}

\def \sut{{\textsf{A\small{llmost\,}\,U\small{pper}\,T\small{riangular}}}}
\def \slt{{\textsf{A\small{llmost\,}\,L\small{ower}\,T\small{riangular}}}}

\newcommand{\overbar}[1]{\overline{#1}}

\def \E{\mathbb{E}}
\theoremstyle{plain}

\newtheorem{lem}{Lemma}
\newtheorem{defi}[lem]{Definition}
\newtheorem{pro}[lem]{Proposition}
\newtheorem{theo}[lem]{Theorem}

\newtheorem{rem}[lem]{Remark}
\newtheorem{exa}{Example}

\newtheorem{cor}[lem]{Corollary}

\numberwithin{equation}{section}
\numberwithin{lem}{section}
\numberwithin{exa}{section}

\def \bls{{\tiny $\blacksquare$ \normalsize }}

\DeclareMathOperator{\Ker}{Ker}

\def \1{\textbf{1}}

\def \Z{\mathbb{Z}}

\def \bT{{\bf T}}

\def \bpar#1{\left\{\begin{array}{#1} }
\def \epar { \end{array}\right.}

\def \N{\mathbb{N}}
\def \R{\mathbb{R}}

\def \q{{\bf q}}

\def \bT{{\bf T}}

\def \bar{\overline}

\def \ba{\begin{align}}
\def \ea{\end{align}}
\def \be{\begin{eqnarray*}}
\def \ee{\end{eqnarray*}}
\def \ben{\begin{eqnarray}}
\def \een{\end{eqnarray}}

\def \beq{\begin{equation}}
\def \eq{\end{equation}}

\def \build#1#2#3{\mathrel{\mathop{\kern 0pt#1}\limits_{#2}^{#3}}}

\def \ba{{\bf a}}

\def \captionn#1{\begin{center}\begin{minipage}{17cm}\sf\caption{\small \textsf{#1}}\end{minipage}\end{center}}

\def \dis{\displaystyle}

\def \eref#1{(\ref{#1})}

\def \P{{\mathbb{P}}}

\def \imp{\Rightarrow}

\def \l{\left}

\def \r{\right}
\def \sous#1#2{\mathrel{\mathop{\kern 0pt#1}\limits_{#2}}}
\def \sur#1#2{\mathrel{\mathop{\kern 0pt#1}\limits^{#2}}}
\def \eqd{\sur{=}{(d)}}

\def\cro#1{\llbracket #1\rrbracket}
\def\croc#1{\rrbracket#1\llbracket}
\def \cror#1{\rrbracket#1\rrbracket}
\def \crol#1{\llbracket#1\llbracket}

\newcommand{\compact}{ \topsep0pt   \itemsep=0pt   \partopsep=0pt   \parsep=0pt}

\newcounter{c}
\def \bir{\begin{itemize}\compact \setcounter{c}{0}}
\def \itr{\addtocounter{c}{1}\item[($\roman{c}$)]} 
\def \eir{\end{itemize}\vspace{-2em}~}

\newcounter{d}
\def \bia{\begin{itemize}\compact \setcounter{d}{0}}
\def \eia{\end{itemize}\vspace{-2em}~}

\newcounter{b}
\def \bi{\begin{itemize}\compact \setcounter{b}{0}}

\def \ei{\end{itemize}\vspace{-2em}~}

\def \DD{{\sf D}}
\def \MM{{\sf M}}
\def \bma{\begin{pmatrix}}
\def \ema{\end{pmatrix}}

\def \Id{{\sf Id}}

\def \U{{\sf U}}

\def \M{{\sf M}}

\def \Id{{\sf Id}}

\def \U{{\sf U}}

\def \M{{\sf M}}

\def \E{\mathbb{E}}
\theoremstyle{plain}
\numberwithin{equation}{section}
\numberwithin{lem}{section}

\def \1{\textbf{1}}

\def \Z{\mathbb{Z}}

\def \bT{{\bf T}}

\def \bpar#1{\left\{\begin{array}{#1} }
\def \epar { \end{array}\right.}

\def \N{\mathbb{N}}
\def \R{\mathbb{R}}

\def \q{{\bf q}}

\def \bT{{\bf T}}

\def \bar{\overline}

\def \ba{\begin{align}}
\def \ea{\end{align}}
\def \be{\begin{eqnarray*}}
\def \ee{\end{eqnarray*}}
\def \ben{\begin{eqnarray}}
\def \een{\end{eqnarray}}

\def \beq{\begin{equation}}
\def \eq{\end{equation}}

\def \build#1#2#3{\mathrel{\mathop{\kern 0pt#1}\limits_{#2}^{#3}}}

\def \ba{{\bf a}}

\def \captionn#1{\begin{center}\begin{minipage}{17cm}\sf\caption{\small \textsf{#1}}\end{minipage}\end{center}}

\def \dis{\displaystyle}

\def \eref#1{(\ref{#1})}

\def \P{{\mathbb{P}}}

\def \imp{\Rightarrow}

\def \l{\left}

\def \r{\right}

\def \uniform#1{{\sf Uniform}(#1)}

\usepackage{mathbbol}
\usepackage{bbm}
\def \clup{almost upper-directed\xspace}
\def \cld{almost lower-directed\xspace}

\def \Cld{Almost lower-directed\xspace}

\def \Motzkin{Mötzkin\xspace}

\def \AUD{{\sf AUD}\xspace}
\def \ALD{{\sf ALD}\xspace}

\def \p{\mathbbm{p}}
\def \q{\mathbbm{q}}
\def \s{{\sf s}}
\def \Q{\mathbb{Q}}
\def \root{\varnothing}
\def \uU{{}^{u}{\U}}

\begin{document}

\maketitle

\begin{abstract} The transition matrix of a Markov chain $(X_k,k\geq 0)$ on a finite or infinite rooted tree is said to be \clup{} if, given $X_k$, the node $X_{k+1}$ is either a descendant of $X_k$ or the  parent of $X_k$. It is said to be  \cld{} if given $X_k$, $X_{k+1}$ is either an ancestor of $X_k$ or a child of $X_k$. \par
 These models include nearest neighbor Markov chains on trees.
 Under an irreducibility assumption, we show that every \clup{} transition matrix on infinite (locally finite) trees has some invariant measures. An invariant measure $\pi$ is expressed thanks to a determinantal formula. We give general explicit criteria for recurrence and positive recurrence. An efficient algorithm (the leaf addition algorithm) of independent interest allows  $\pi$ to  be computed on many trees, without resorting to linear algebra considerations. \par  
 Flajolet, in a series of papers (starting from \cite{Flajolet}), provided some relations between continuous fractions, generating functions of weighted \Motzkin paths, and used them in connection with the analysis of birth and death processes. These fruitful representations made it possible to establish many formulae for continuous fractions. Analogous considerations appear here: this type of study can be extended to weighted paths on trees, whose generating functions can also be expressed, this time in terms of multicontinuous fractions.
\end{abstract}

\section{Introduction}
A rooted tree $(t,r)$ is a graph, finite or infinite, but always locally finite throughout this paper, connected and having no cycle, in which the root $r$, a node of $t$, is distinguished. In this paper, we are interested in Markov chains on trees. When the tree is finite, the study of such a Markov chain can be done with the standard toolbox of Markov chains on finite spaces. When the tree is infinite, it is generally difficult to find invariant distributions, to decide on recurrence or transience, on positive recurrence. \par
  In the literature, one finds many studies of nearest neighbors random walks on trees, often with transition matrices depending on the degree of the current node, and very few papers discussing general models of Markov chains on trees. \par
  The aim of this paper is to present a large class of transition matrices, much larger than the class of nearest-neighbors random walks, that are accessible to computation and analysis.
  These Markov transition matrices will be called either \clup{} (\AUD for short), when only transitions to descendants and the parent are allowed, or \cld{} (\ALD for short), when only transitions to ancestors and to children are allowed.
\begin{figure}[htbp]
\centerline{\includegraphics{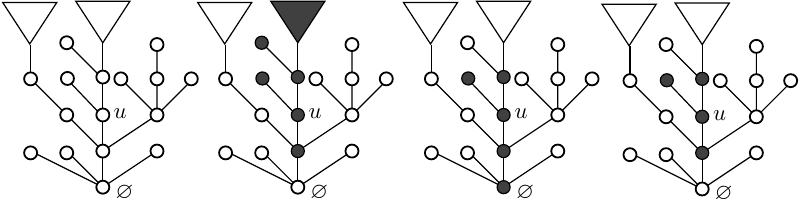}}
\captionn{\label{fig:tree-clup} $\root$ is the root, and $u$ a node. Triangles figure finite or infinite subtrees, and disks, nodes. On the second picture, only the transitions from $u$ toward dark nodes are possible for \clup{} transitions matrices. On the third pictures,  only the transitions from $u$ toward dark nodes are possible for \ALD{} transitions matrix. For nearest-neighbor Markov chains, only, the neighbors of $u$ are accessible (fourth picture).}
\end{figure}
  At the intersection of these two classes are nearest neighbor Markov chains, probably the  most studied class of Markov chains on trees. The class \AUD is easier to study: we provide for this class formula and algorithm to compute the invariant measures, we provide criteria to decide recurrence or positive recurrence; the algebra in play is rich enough: we shed some light on the combinatorial points of view allowing to provide determinantal formula for the invariant distribution, multicontinuous fractions representation of the Green functions, etc. \par
Before entering into the core of the paper, we think that it is useful to say a few words about Markov chains on $\N$, and particularly about those with almost triangular transition matrices. The space $\N$ can be seen as an infinite tree with no branching points, and many phenomena appearing on $\N$ will have some counterparts on trees: in fact, the class of transition matrices that we will study in the sequel, and for we will provide computations and criteria, is a generalization of these ``almost triangular transition matrices''.

\subsection{Content of the paper}

The paper is organized as follows: In \Cref{sec:W}, we explain that \AUD and \ALD Markov chains on the line (which is an infinite tree without branching points) are simply almost upper and lower triangular Markov chains: already for these models, a determinantal formula for the invariant distribution appears. We formally define trees in \Cref{sec:TF}, as well as the two classes \AUD, \ALD and random walks on trees. We give a necessary and sufficient condition for the irreducibility of these chains. Finally, in \Cref{sec:TE} we recall the formula for the invariant measure of an irreducible Markov kernel on a finite state space: again, a determinantal formula arises, which is useful to have in mind before entering to the main new result sections.

  In \Cref{sec:AUD1}, we prove that all \AUD irreducible transition matrices admit at least one invariant measure, which  we call the $h$-invariant measure, which has an explicit determinantal formula (\Cref{theo:1.3}). This allows us to provide a criterion for deciding positive recurrence. In \Cref{pro:rec}, a first criterion for recurrence of \AUD chains is given. Section \ref{sec:qdgg}, more simple  criteria are given, when the number of ends of the tree is finite (the projections of the chains onto the ends serve as a tool in this matter).

  In \Cref{sec:aud}, we present an alternative algorithm to compute the $h$-invariant measure, which we call the leaf addition strategy. We then discuss the existence of other invariant measures, and more general spectral properties are obtained. Among other things, we will see in \Cref{sec:eig} that irreducible \AUD Markov chains on infinite trees without leaves have a complete spectrum (all elements of $\C$ are left eigenvalues), while this is not generally the case on infinite trees with some leaves (in any case, the set of complex numbers that are not left eigenvalues is at most countable).
  
  In \Cref{sec:RWT}, we discuss random walks on trees, and in particular, random walks on critical Galton-Watson trees conditioned on non-extinction: a model of an infinite tree with a single end. In  \Cref{pro:sqfd}, we provide a  recurrence criterion for these chains when the transition coefficients $\U_{u,v}$ from a node $u$, are a function of the degree of $u$.  
  In \Cref{sec:CLD}, we say a few words about \ALD transition matrices. They are in general more difficult to study than $\AUD$, except when they are reversible with respect to some invariant measure, and in this case, the reverse time is an \AUD transition matrix, which is easier to handle.  \Cref{sec:CAA} highlights two combinatorial relations involving \AUD matrices: the link between the generating functions of weighted paths (corresponding to Markov chain trajectories) and multicontinuous fractions (à la Flajolet), a random walk on the rational numbers (using the Stern-Brocot tree to encode the possible transitions). Finally, in \Cref{sec:example}, we compile a series of examples to illustrate the application of our results, which are referred to throughout the article.

\subsection{Warmup with the simplest infinite tree: the line}
\label{sec:W}
This paper can be seen as a continuation of \cite{MR4669754} by the same authors (reading \cite{MR4669754} is not a prerequisite for understanding this article), in which Markov chains on $\N$ with transition matrices that are almost lower (\slt) or almost upper (\sut) triangular are studied. A Markov chain $(X_k,k\geq 0)$ on $\N$ with a \sut{} transition matrix has its increments $X_k-X_{k-1}$ greater than or equal to $-1$,  while \slt{} Markov chains 
 have increments less than or equal to $+1$.
 If we consider $\N$ as an infinite tree, rooted at zero, where $i+1$ is the only child of $i$ (there are no branching points), we can see that the \sut{} transition matrices on this tree correspond to \AUD transition matrices, while \sut{} transition matrices  correspond to \ALD transition matrices.

 Thus, each of these models generalizes birth and death processes, whose transition matrices are simultaneously \slt{} and \sut.\par
In  \cite{MR4669754}, it is shown, among other things, that \sut{} transition matrices have a simpler behavior than \slt{}. 
 For example, they have a single invariant measure (up to a multiplicative constant), which is not the case for \slt{} for which the cone generated by positive invariant measures can have any dimension  ranging from 0 to $+\infty$. Characterizations for recurrence and positive recurrence are also given.
~\par
The difference in behavior between \sut{} and \slt{} Markov chains somehow emerge from the difference in nature of the algebraic equations solved by the invariant measure $\rho = \rho M$.  
The solution of the equation $\rho = \rho M$ is immediate in the \sut{} case, in which it is a triangular system, while in the lower case, the equation expresses $\rho_j$ as a linear combination of $(\rho_{j+k},k\geq 1)$, which gives rise to the variety of behaviors discussed above.

Hence the variety and complexity of behaviors discussed above (and in \cite{MR4669754}) for \sut{} and \slt{} Markov chains, exist on tree Markov chains.  

\subsection{Tree formalism. }
\label{sec:TF}

  A rooted tree $T$ is a subset of the set of finite words ${\cal U}:=\{\root\}\cup \bigcup_{k\geq 1}\mathbb{N}^k$ on the alphabet $\mathbb{N}$ such that: \\
$(i)$ it contains the element $\root$ (the empty word), seen as the root,\\
$(ii)$ if the word $u_1\cdots u_h$ is in $T$ then $T$ contains each of its prefixes $u_1\cdots u_k$ for $k\leq h$ (including $\root$),\\
$(iii)$ if the finite word $u_1\cdots u_h$ is in $T$ then so does $u_1\cdots u_{h-1}i$ for all $i$ in $\{1,\cdots,u_h\}$.\\
For $u\in T$, the set $\cro{\root,u}$, called the set of ancestors of $u$ is defined as the set of prefixes of $u$ (including $u$ and $\root$). The set $T_u:=\{w \in T~: u \in \cro{\root,w}\}$ is the set of descendants of $u$ in $T$. \\
A word $u\in T$ is seen as a node of $T$, its length $|u|$ is its height; the parent $p(u)$ of a node $u\neq \root$, is its ancestor having height $|u|-1$.
\begin{figure}[htbp]
  \centerline{ \includegraphics{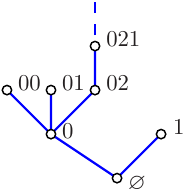}}
 \captionn{Representation of the tree $t=\{\root, 0,00,01,02,021,...\}$. The set of children of $02$ is $c_t(02)=\{021\}$, the nodes $1$, $00$, $01$ are leaves, $\root, 0, 02, 021$ are internal nodes. On this example $\cro{\root,01}=\{\root, 0,01\}$.}
\end{figure}~\\
The set of children $c_T(u)$ of a node $u$ in $T$ is the set of words $ui\in T$ where $i$ is a single letter. From this we see that $|c_T(u)|=\max\{i~:ui\in T\}$ is the number of children of $u$.
 A node with no children is called a leaf, and the set of leaves of $T$ is denoted by $\partial T$ (a finite tree has at least one leaf, but an infinite tree can have zero leaves).

A tree is said to be infinite if its cardinality is infinite, and finite otherwise. 

An end is an injective path starting at the root and isomorphic to the discrete line $\N$. An infinite tree can have from one end to an infinite uncountable number of ends.

\paragraph{Conventions.}~\\Throughout the paper:\\
$(i)$ All trees  are considered to be locally finite (even when unspecified), i.e., the number of children of each node is finite.
\\
$(ii)$ The set of non-negative measures over a tree $T$ (equipped with the power set $\sigma$-field) is denoted by ${\cal M}_T^{+}$; denote by $\sim$ the proportionality equivalence relation (we write $\mu \sim \nu$ when $\mu$ and $\nu$ are proportional).\\
$(iii)$ The identity matrix is denoted $\Id$, whatever is its size, and the size will be clear from the context. For example, if $A$ is a matrix, in the expression $\Id-A$, $\Id$ is the identity matrix that has the same size as that of $A$.\\
$(iv)$ Recurrence/transience/positive recurrence, will be used to qualify Markov chains and transition matrices\\
$(v)$ If $\U$ is a transition matrix, we call $\U$-Markov chain a Markov chain having $\U$ as transition matrix,\\
$(vi)$ We  will say that ``$v$ is a $\lambda$-left (or right) eigenvector of $A$'', as a shorthand for  ``$v$ is a left (or right) eigenvector of $A$ associated with the eigenvalue $\lambda$''.\medskip

When dealing with the set of invariant measures of a transition matrix, it is common to work in ${\cal M}_T^{+,\sim} =  {\cal M}_T^{+}/\sim$, where uniqueness actually means indeed uniqueness up to a multiplicative factor. The set of invariant measures is said to be $d$-dimensional if ${\cal M}_T^{+}$ generates a vector space of dimension $d$.

\begin{defi} Let $T$ be a finite or infinite tree (see \Cref{fig:tree-clup}).
  \begin{itemize}
\item[\bls]  A transition matrix $\U=\bma \U_{u,v}\ema_{u,v\in T}$ is said to be \underbar{\clup{}}, ``\AUD'' for short, if 
\[\forall u,v \in T, \quad\U_{u,v}>0 \imp (v \in \{ p(u)\}\cup T_u)\]
that is, if either $v$  is the parent $p(u)$ of $u$  or a descendant of $u$.
Since $u\in T_u$, $\U_{u,u}$ is allowed to be positive. 
\item[\bls] 
 A transition matrix $\DD=\bma \DD_{u,v}\ema_{u,v\in T}$ is said to be \underbar{\cld{}}, ``\ALD'' for short, if 
 \[\forall u,v \in T, \quad \DD_{u,v}>0 \imp ( v \in c_T(u)\cup \cro{\root,u})\]
 that is if either $v$  is a child of $u$ or one of its ancestor.
Since $u\in\cro{\root,u}$, $\DD_{u,u}>0$ is also allowed.
\item[\bls]   A transition matrix $\MM=\bma \MM_{u,v}\ema_{u,v\in T}$ is said to be the transition matrix of a \underbar{random walk} on $T$, if $\;\MM_{u,v}>0$ implies $v\in\{u,p(u)\}\cup c_T(u)$ (it is simultaneously \clup and \ALD).\end{itemize}
\end{defi}
A Markov chain $(X_k,k\geq 0)$ with an \AUD transition matrix, is somehow (discrete) lower semi-continuous: it is allowed to make jumps of size 1 (at most)  towards the root, and unbounded jumps towards the descendants of the current node.\par
A Markov chain  $(X_k,k\geq 0)$ with an \ALD transition matrix, is somehow  (discrete) upper semi-continuous: jumps to the descendants are reduced to the children, while jumps to the ancestors are unbounded.
\begin{lem} \label{lem:irr}[Irreducibility criteria]
  \bir
  \itr An \AUD transition matrix $\U$ is irreducible iff, for all $u\neq \root$, $\U_{u,p(u)}>0$ and, for each node $v\neq \root$, there exists a strict ancestor $u$ of $v$ and an element $w$ in $T_v$ such that $\U_{u,w}>0$.
\itr An \ALD transition matrix $\DD$ is irreducible iff, for all $u\neq \root$, $\DD_{p(u),u}>0$ and for each node $v$, there exists a descendant $w$ of $v$ and a strict ancestor $u$ of $v$, such that $\DD_{w,u}>0$.~
\itr The transition matrix of a random walk on $T$ is irreducible iff for each $u\neq \root$, $\MM_{u,p(u)}>0$ and $\MM_{p(u),u}>0$. \eir
\end{lem} 
\begin{proof} $(i)$ The hypotheses are necessary to come back to $u$ starting from $u$, for all $u$. Let us see why they are sufficient. Each vertex is accessible  (in an arbitrary number of steps) from any of its descendants. So, it suffices to show that every node $v$ has a descendant that is accessible.
By contradiction, suppose that $v$ is a node other than $\root$ that is  not accessible from $\root$. In this case, none of the descendants of $v$ is accessible either; since there is an ancestor $u$ of $v$ such that $\U_{u,v}>0$,
we deduce that $u$ is also inaccessible. By iteration, one sees that there cannot be a "smallest" inaccessible vertex, so that the root is inaccessible which is clearly false for any starting vertex.\\
$(ii)$ The conditions are necessary in order to access a given node $v$ from its ancestors, and to climb up in the tree. Let us see why they are sufficient. Under an \ALD transition matrix, one can traverse from a node $u$ to any of its descendants $v$. It is then sufficient then to show that, given $v\neq \root$ in the tree, there exists a path with positive weight from $v$ to the root. The fact that an ancestor of $u$ is accessible from a descendant $w$ of $v$ (which is accessible from $v$), shows that from $v$, one can access to nodes closer to the root. By iterating (from $u$) one can prove that the root is accessible.
\\
$(iii)$ In the case of random walks, a transition matrix is irreducible iff every edge can be traversed both ways.
\end{proof}

\subsection{Tree expansion of the invariant distribution of Markov chains on finite state space.}\label{sec:TE} The content of this section is classical (and is due to Leighton and R. Rivest \cite{LR}, Aldous \cite{Al90} and Broder \cite{Bro89}; see also \cite{HLT21}, \cite{LFJFM}). 
If $V$ is a finite state space and $K$ an irreducible transition matrix over $V$, define the oriented graph $G=(V,E)$ where an edge $e=(v_1,v_2)$ is in $G$ if and only if $K_{v_1,v_2}>0$. 
The unique invariant distribution $\pi$ of $K$ is
\ben\pi(r) = {\sf Cste}. \det\l( \Id-K^{(r)}\r),~~~\textrm{ for all } r\in V,\een
where $K^{(r)}$ is the matrix $K$ in which the $r$th row and column has been suppressed. The ``graph view'' is not needed for this statement, but it allows to state that
\ben\label{eq:grte}
\pi(r) = {\sf Cste}.\sum_{(t,r) \in {\sf SP}(G,r)}~~ \prod_{e=(e_1,e_2)\in E(t,r)} M_{e_1,e_2}
\een
where ${\sf SP}(G,r)$ is the set of spanning tree of $G$ rooted at $r$, and the edge set $E(T,r)$ for a spanning tree $(T,r)$ contains all the edges of $T$ directed towards $r$ (that is the vertex $e_2$ is closer to $r$ than the vertex $e_1$).

More generally, given an oriented graph $G=(V,E,w)$ where $w:V^2\to \C$ is a function that assigns a weight to each oriented edge (a pair in $V^2\setminus E$ has weight $0$), the \textit{Laplacian matrix} of the graph is
\ben\label{eq:hterfdsq}
{\sf Laplacian}(w)= \bma 1_{i=j}\l(\sum_{a\in V} w_{i,a}\r) - w_{i,j} \ema_{i,j\in V}, \een the total weight of the set of spanning trees rooted at some node $r\in V$ is
\ben {\sf Weight}(V,w,r)=\det ({\sf Laplacian}(w)^{(r)}),\een
(this is Kirchhoff matrix tree theorem \cite{K1847}, see Zeilberger \cite{DZ}) so that, again, when $K$ is an irreducible transition matrix on a finite space,
\ben\label{eq:grte2}
\pi(r) = {\sf Cste}.\det ({\sf Laplacian}(K)^{(r)}).
\een

\section{\AUD transition matrices : invariant measure, recurrence and transience}
\label{sec:AUD1}
\subsection{An explicit invariant measure}

For a transition matrix $\U$, a node $a$, and a set of nodes $B$, we let $\U_{a,B}=\sum_{b\in B} \U_{a,b}$. Recall that $T_a$ is the subtree of $T$ rooted at $a$. For a node $u$ in $T$ and $v\in\cro{\root,p(u)}$, the successor of $v$ in the direction of $u$ is denoted by $\s(v,u)$ (this is the unique child of $v$  on the path $\cro{\root,u}$).
\begin{theo}\label{theo:1.3}
Consider  $\U$  an irreducible \AUD transition matrix of a finite or infinite tree $T$.  For any node $u$ of $T$, consider the weighted graph $({}^uG,\uU)$ with set of nodes ${}^uV:=\cro{\root,u}$, and in which the weight of
\[\uU_{a,b} = (\U_{a,T_b}-\U_{a,T_{\s(b,u)}})\1_{b\in \cro{p(a),p(u)}} +\1_{b=u}\U_{a,T_u}\]
is defined for all $a\in\cro{\root,p(u)}$ and $b\in\cro{\root, u}$. Set
\ben\label{eq:pi}
\pi(u) = \pi(\root)\frac{{\sf Weight}( \cro{\root,u}, {}^{u}\U, u)}{{\dis\prod_{v \in \cror{  \root,  u}} \U_{v,p(v)} }}=\pi(\root)\frac{\det ((\Id-\uU)^{(u)})}{\dis\prod_{v \in \cror{  \root,  u}} \U_{v,p(v)} }.
\een
\bir
\itr The measure $\pi$ is an invariant measure of $\U$ (with positive coordinates).
\itr An \AUD transition matrix is positive recurrent if and only if $\sum \pi(u)<+\infty$.\eir
\end{theo}
\noindent 

\bls We call the measure $\pi$, as defined in \eref{eq:pi}, the \textsf{h}-invariant measure, since, as we will see in \Cref{seq:h-t}, it can be seen as the limiting invariant measure of a kind of \textsf{h}-transform of $\U$.\\
\bls As we will see in \Cref{sec:qfgrht}, there exist infinite trees $T$ and irreducible \AUD{} matrices $\U$ defined on them, admitting many invariant measures in ${\cal M}_T^{+,\sim}$ (finite trees are simpler and will be treated at the beginning of the proof). However, it is well known that an irreducible transition matrix $\U$ which is recurrent (including then the positive recurrent case), admits a unique invariant measure (see e.g. Brémaud \cite[Theo. 2.2]{Bremaud}): In this case, by \Cref{theo:1.3}, this invariant measure is $\pi$. Hence $(ii)$ is just a standard consequence of $(i)$.\\
\begin{rem} Observe that for $a\in\croc{\root,p(u)}$, $\uU_{a,p(a)}=\U_{a,p(a)}$ and for $a\in\cro{\root,p(u)}$, $\uU_{a,a}=\U_{a,T_a}-\U_{a,T_{\s(a,u)}}$. In fact, given the proof, it could have been natural, for combinatorial reasons, to take  $\uU_{a,a}=\U_{a,a}$ or $\U_{a,a}=0$ instead, since these quantities follow naturally from the analysis. However, when computing the Laplacian matrix, the value of the diagonal coefficients $\U_{a,a}$ are "removed" by the diagonal correction, so that, we can modify them as we wish. With the choice we took, we have
\[D_a{(u)}=\U_{a,T_u}+\sum_{b \in \cro{p(a),p(u)}}\l(\U_{a,T_b}-\U_{a,T_{\s(b,u)}}\r) =1,\]
and then ${\sf Laplacian}( {}^{u}\U)^{(u)}= \det ((\Id-\uU)^{(u)})$.\\
\bls We do not need to define the transition  $(\uU_{u,b}, b \in\cro{\varnothing,u})$ ``out of $u$", since in the calculation of the right hand side of \eref{eq:pi} the lines and columns corresponding to $u$ are suppressed. Nevertheless we could have taken $\uU_{u,p(u)}=\U_{u,p(u)}$ and $\uU_{u,u}=1-\U_{u,p(u)}$, so that $\Id-\uU$ would be the Laplacian matrix of $\uU$.
 \end{rem} 
 
\begin{rem}[Inherent complexity.]  
  When one deals with irreducible Markov chains on an infinite graph, all transitions are important: changing the transitions from a single state can transform a positive recurrent chain into a non-positive recurrent chain. For example, the random walk on $\mathbb{N}$, reflected at zero, with i.i.d. steps $-1$ with probability $2/3$ and $+1$ with probability $1/3$ is positive recurrent. However, if one replaces only the transition from 0 and takes instead $M_{0,i}=1_{i\geq 1}(6/\pi^2)/i^2$, the chain is no longer positive recurrent.  Since invariant measures must depend algebraically on all transitions, all criteria of recurrence, all formulae providing the invariant measures (of left or right eigenvectors) must depend on all transitions. In this matter, beyond their apparent complexity, formulas such as \eref{eq:pi}, which use products of transitions along branches, or determinants of finite matrices, are the simplest formulas that allow each coefficient of the transition matrix to play a role.
\end{rem}

\subsection{A general criterion for recurrence}

\begin{pro}\label{pro:rec} Let $\U$ be an \AUD irreducible transition matrix on an infinite tree $T$, and $(X_k,k\geq 0)$ a $\U$-Markov chain. For any set of nodes $A$, denote by
  \[\tau_{A}=\tau_A(X)=\inf\l\{k: X_k\in A\r\}\] the hitting time of $A$; for a single node, as the root $\root$, we write $\tau_{\root}$ instead of $\tau_{\{\root\}}$. Let $T_{\geq h}=\{u \in T, |u|\geq h\}$ be the subset of $T$ consisting of the nodes at distance $\geq h$ from $\root$.
  \bir
	\itr  $\U$ is recurrent on $T$ if and only if for each node $i$ of the first generation ($i\in c_T(\root)$), we have
	\ben\label{eq:rec}\lim_{b \to +\infty} \P(\tau_{\root} < \tau_{T_{\geq b}}~|~X_0=i)=1.\een
	\itr 
Let  $T_{i,< h}=\{u \in T_i,~|u|<h\}$ be the restriction of the subtree $T_i$ to the $h-1$ first levels (of $T$), and $U_{[i,<h]}$ be the matrix obtained by keeping only the rows and columns of $\U$ corresponding to the elements of $T_{i,< h}$ (and $\U^{(i)}_{[i,<h]}$ for the same matrix, deprived from the row and column indexed by $i$). We have:
	\ben\label{eq:redqc}
	\P(\tau_{\root} <  \tau_{T_{\geq h}}~|~X_0=i)=\frac{\det(\Id-\U^{(i)}_{[i,<h]})}{\det(\Id -\U_{[i, <h]})} \U_{i,\root}. 
	\een
\eir\end{pro}
\begin{proof}
  \bir \itr If $X$ is recurrent, then starting from any $i\in C_T(\root)$, the return time  $\tau^+_i$ to $i$ is finite a.s. After each passage at $i$, $X$ has probability $\U_{i,\root}$ to go to $\root$ in one step, otherwise, it makes a step inside the subtree $T_i$. 
  As a random variable over $\mathbb{N}$, $\tau^+_i$ is tight: for all $ \varepsilon>0$, there exists an integer $B$ such that $\P_i(\tau_i^+>B)\leq \varepsilon$. Since, going from a node at height $\geq b$ to $i$ takes at least $b-1$ steps, we see that for $b$ large enough, $\P_i(\tau_i^+<\tau_{T_{\geq b}})\geq 1-\varepsilon$. From here, the argument is routine: since  $\U_{i,\root}<1$,  up to increase $b$, this implies also $q_i(b):=\P_i(\tau_i^+<\tau_{D\geq b}~|~X_1\in T_i)\geq 1-\varepsilon$ where $T_i$ is the subtree rooted at the children $i$ of the root. 
Hence, $\P(\tau_\root^+<\tau_{D\geq b}~|~X_0=i)= \sum_{k\geq 0}\U_{i,\root}(1-U_{i,\root})^{k}q_i(b)^k$ (since each failure to go from $i$ to $\root$, arises with probability $(1-U_{i,\root})$ and another try will follow with probability $q_i(b)$. Since $q_i(b)$ can be taken as close to 1 as wished, we get the result.\par
Reciprocally, it suffices to prove that if \eref{eq:rec} holds, then $\P(\tau_\root^+<+\infty~|~ X_0=\root)=1$ where $\tau_\root^+$ is the return time to $\root$.
 Let us assume that $\P(\tau_\root^+<+\infty~|~ X_0=\root)<1$. By the law of total probability, this implies that there exists $u\in \cup_i T_i = T\setminus\{\root\}$,  $\P(\tau_\root^+<+\infty~|~ X_1=u, X_0=\root)<1$. The node $u$ belongs to one of the trees $T_i$, say $T_{i^\star}$, and is then, by irreducibility, reachable from $i^\star$ by a finite path $(v_0=i^{\star},\cdots,v_k=u)$ for some $k\geq 0$, having a positive weight, and then again by the law of total probability, it appears that  $\P(\tau_\root^+<+\infty~|~ X_1=u, X_0=\root)<1$ implies that $\P(\tau_i^{+}<\infty~|~X_0=i)<1$, and then, by the reasoning above, \eref{eq:rec} can not hold.
\itr We need to sum the weight of all paths $(u_0,\cdots,u_\ell=\root)$ of any length , starting at $u_0=i$, hitting $\root$ for the first time at the end, staying in $T_i$ deprived of $T_{\geq h}$ all along, and weighted by $\prod_{j=1}^{\ell} \U_{j-1,j}$. All these paths ends by the step $i\to \root$, and the rest is a heap of cycle whose maximal pieces is incident to $i$ the result follows from classical results in combinatorics  (see the proof of Theorem 2.8 in \cite{LFJFM} by the same authors; see also e.g.  Viennot \cite{VX}, Krattenthaler \cite{CK}).
\eir
\end{proof}

\begin{proof}[Proof of \Cref{theo:1.3}]
	
	\bls \underbar{Finite tree case.} Assume first that $\U$ is an \AUD irreducible transition matrix on a finite tree $T$.
	In this finite state case, there is uniqueness of the invariant measure and positive recurrence.
	\par
	It remains to prove that \eref{eq:pi} gives the invariant measure of $\U$.
	We will use \eref{eq:grte} and \eref{eq:grte2}, and then define for $u\in T$, the set of spanning trees  ${\sf SP}(T,E_\U,u)$, with root $u$, on the graph with vertex set $V=T$, and set of edge sets corresponding to (almost upper directed) possible transitions  $E_{\U}:=\{(i,j)\in T, \U_{i,j}>0\}$. The edges of a tree are oriented towards its root: thus, for a spanning tree rooted at $u$, there is an edge going out from each vertex, except from $u$.

    $\star$ An illustration in \Cref{spannn.fig} could help the reader to understand more quickly the argument.

    \begin{figure}[htbp]
	\centerline{\includegraphics{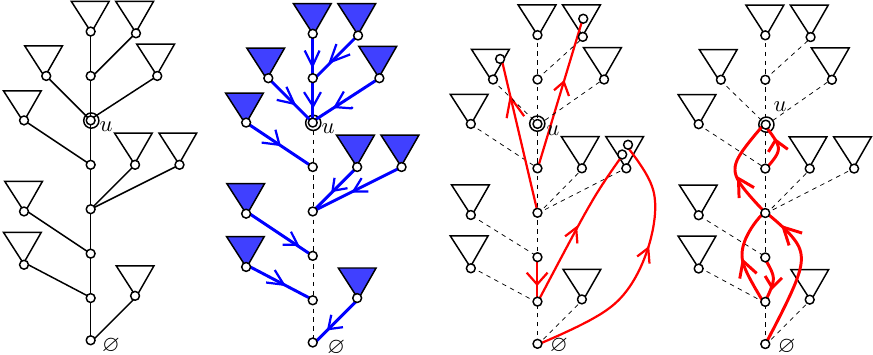}}
	\captionn{\label{spannn.fig}Each triangle represents a finite or infinite subtree. On the second picture, the blue edges (including those, not drawn, in the subtrees) are all oriented downwards: these edges are in all spanning trees rooted at $u$: the starting point of each oriented edge is any node which is not in $\cro{\root,u}$. The set of blue edges form connected components, that are subtrees, each of them rooted a different vertices of $\cro{\root,u}$ (this is the forest $F_u=(f_v,v\in\cro{\root,u})$). On the third picture, an example of what could be the set of edges  going out from $\cro{\root,u}$ of a spanning tree rooted at $u$: either directed to the parent, or to a node among their descendants. On the last picture, the ``spanning tree condition'' on $\cro{\root,u}$. If we redirect each red edge toward the root of the blue component containing its second extremities, this forms a (red) spanning tree of $\cro{\root,u}$: this condition is necessary and sufficient for the blue edges and red edges to form, together, a spanning tree of the global tree, rooted at $u$.}
    	\end{figure}
     Consider $\cro{\root,u}$ the ancestral line of $u$ in $T$.  
We claim that ${\sf ForcedEdges}(u)$, the set of edges  $(v,p(v))$  for all $v\in (T\setminus \cro{\root,u})$, is a subset of the edge set of any $(\tau,u)$ in ${\sf SP}(T,E_\U,u)$. Indeed since the edges of $(\tau,u)$ are oriented toward $u$, each edge  $(v,p(v))\in{\sf ForcedEdges}(u)$ is needed to get out of $T_v$. Now, color the edges of ${\sf ForcedEdges}(u)$  in blue, together with the vertices they contain. 
Each blue connected component $f$ is a tree, and contains a unique element $v$ of $\cro{\root,u}$, which is the root of $f$.  Thus, each node $v\in\cro{\root,u}$ is the root of a blue tree $f_v$, where, $f_v$ may be reduced to its root. Each node $w$ of $T$ belongs to a single tree $f_{R(w)}$ of the forest  $F_u := (f_v,v\in\cro{\root,u})$.
In other words, $R(w)$ is the identity of the root of this tree, a node of $\cro{\root,u}$.

 Let us now characterize the other edges of a spanning tree $(\tau,u)\in{\sf SP}(T,E_\U,u)$, those starting from the nodes $\crol{\root,u}$. 
  \begin{lem}\label{lem:ctree}
    Consider a set of edges, with one outgoing edge from each the elements of $\crol{\root,u}$~:
    \[{\sf AddEdges}:=((v,x(v)), v \in \crol{\root,u})\] where $x(v)\in T$ for all $v$. 
		The set ${\sf AddEdges}$ union with  ${\sf ForcedEdges}(u)$  is the set of edges of an element $(\tau,u)$ of ${\sf SP}(T,E_\U,u)$ if the tuple $\l[(v, R(x(v))),v \in \crol{\root,u}\r] $ is a spanning tree $(\tau',u)$ of $\cro{\root,u}$. 
\end{lem}
\begin{rem}\label{rem:qgsr} The "if" in the lemma becomes an "if and only if" if all the edges weight $\U_{v_1,v_2}$ are positive between each node $v_1$ and its descendant $v_2$ (by hypothesis $\U_{v,p(v)}$ are always positive).
\end{rem}
	
\begin{proof} Given that the blue edges are already present in the forest, we need to understand the structure of the other edges of  $(\tau,u)$, i.e., those whose starting point in $\crol{\root,u}$. 
		From each node $v\in\crol{\root,u}$, there is a single outgoing edge $(v,x(v))$ and $x(v)$ must belong to an element $f_{R(x(v))}$ of the $(f_w, w\in \cro{\root,u})$ (since these trees form a partition of $T$). The coefficient $\U_{v,x(v)}$ must be positive, so $x(v)$ is either a descendant of $v$, or its parent; from $x(v)$, following the outgoing blue  edges forms a path leading to $R(x(v))$, but these edges are not the subject of this discussion, since they are already fixed (they are blue). Intuitively, since there is a fixed path from all the nodes of $f_{R(x(v))}$ toward $R(x(v))$, one can contract any tree $f_w$ and consider it as a vertex $\bar{w}$ (which can be identified with $w$), locally, in the proof. Notice that if several extreme points $x(v_1)$ and $x(v_2)$ belong to the same tree $f_w$, then, since in $f_w$ the blue paths are directed to the root $w$, these paths together do not form a cycle (the contraction does not destroy any cycle). The contraction is coherent with this point of view (the edges are then directed to $\bar{R(x(v))}$).\par
		Once the contractions have been done, the new edges $(\bar{v},\bar{R(x(v))})$ must form a spanning tree $(\tau',u)$ of $(\bar{v},\bar{v}\in \cro{\root,u})$ (because if it is not the case, the union of ${\sf AddEdges}$ union with  ${\sf ForcedEdges}(u)$ would be disconnected).
      \end{proof}

      Extract from a spanning tree $(\tau,u)$, the set of edges
      \[{\sf SetRedEdges}(\tau,u)=\l\{(v,x(v)) \in E(\tau,u), v\in \crol{\root,u}\r\}\] going out from $\crol{\root,u}$.
      Following \Cref{lem:ctree} and \Cref{rem:qgsr}, we denote by
      \[{\sf SetOfSetRedEdges} = \l\{\l[ (v, x(v)), v \in \crol{\root,u}\r] \textrm{ s.t.} \l[(v,R(x(v))),v \in \crol{\root,u}\r] \textrm{ is a spanning tree of }\cro{\root,u}\r\}.\]
As said in \Cref{rem:qgsr}, in general the set of all possible ${\sf SetRedEdges}(\tau,u)$ is only contained\linebreak in ${\sf SetOfSetRedEdges}$, because  ${\sf SetOfSetRedEdges}$ contains elements, that contain ``red edges'' with null\linebreak weight. However, since we are now dealing with weights, adding elements with null weights amounts to adding negligible sets, and this is what we will do.
       Set
      \ben{\sf RedWeight}(u)&=& \sum_{\l[ (v, x(v)), v \in \crol{\root,u}\r]\atop{\in {\sf SetOfSetRedEdges}}} \prod_{v \in \crol{\root,u}}\U_{v,x(v)},\\
{\sf BlueWeight}(u)&=&\prod_{v\in T\setminus \cro{\root,u}} \U_{v,p(v)}
      \een
      this latter being the total weight of ${\sf ForcedEdges}(u)$. 
      We have that
      \[W(u) := {\sf RedWeight}(u)\times{\sf BlueWeight}(u) \] is the total weight of all spanning trees rooted at $u$ and we need to prove that $W(u)= {\sf Cst}.\pi(u)$ given in \eref{eq:pi}. ~\\	
	In order to complete the proof in the finite case, we need three additional ingredients:\\
	(a) Set ${\sf Factor}(u)=\prod_{v \in \cror{\root,u}}\U_{v,p(v)}$, the denominator in \eref{eq:pi}. Since $\{(v,p(v)):\cror{\root,u}\}$ are the edges we need to add to ${\sf ForcedEdges}(u)$ to get all the edges of $T$ oriented toward $\root$, we have
    \[{\sf Factor}(u)\times{\sf BlueWeight}(u)= \prod_{v\in T \setminus \{\root\}} \U_{v,p(v)}={\sf Cst}.\]
    Hence since we work up to a multiplicative factor, we will use $1/{\sf Factor}(u)$ instead of ${\sf BlueWeight}(u)$, since the first remains under control when the tree is infinite.\\
 	\noindent(b) Let us compute ${\sf RedWeight}(u)$. 
    From the ``contraction'' point of view, an edge from $a\in\crol{\root,u}$ to some node $\bar{b}$ represents the total weight of edges going from $v$ to $f_b$: the tree $f_b$ itself, equals to $T_b\setminus T_{s(b,u)}$ (the subtree $T_b$ deprived from $T_{s(b,u)}$). The weight of all edges starting at $a$ and with second extremity in $f_b$ is then ${}^u\U_{a,b}$. By the matrix tree theorem we then have
	\[
	{\sf RedWeight}(u)=\det ({\sf Laplacian}(^u\U)^{(u)})
	\]
	$(c)$ It remains to justify that we can choose the ${}^u\U_{a,a}$ as we wish: this is standard; in a weighted graphs, the weight of the spanning trees is independent of the weights of the loops, and the user is then free to change them according to his personal motivations... This is what we have done by imposing the values ${}^u\U_{a,a}$ such that the total weights $\sum_{v}{}^u\U_{a,v}=1$ for all $a$. From this choice we find that
	\[
	\det ({\sf Laplacian}(^u\U)^{(u)}) = \det ((\Id-\uU)^{(u)})
	\]
	This concludes the proof for the finite case.
 	\medskip
	
	\noindent\bls \underline{When $T$ is an infinite (local finite) tree, }  write the balance equation at $u$ for an invariant measure $\rho$ (normalized so that $\rho(\root)=1$):
	\begin{align}\label{eq:inv}
		\rho(u)=\sum_{v \in \cro{\root, u}} \rho(v) \U_{v,u}+\sum_{c \in c_T(u)} \rho(c)\U_{c,u}.
	\end{align}
	We want to prove that the measure $\rho=\pi$ given in \eqref{eq:pi} solves this equation.  
	Thus, if we write the invariance equations for all nodes $u$ below a given level $h$, then all  these equations together involve only $\rho$ and transitions up to level $h+1$. Let us cut down this tree! \par
    Consider a transition matrix $\U^{(h+2)}$ on say, the finite tree $T^{(h+2)}$ of height $h+2$ coinciding with $T$ up to this level, and define  $\U^{(h+2)}_{u,v}$ as to be $\U_{u,v}$ if $|v|<h+2$, and $\U^{(h+2)}_{u,v}=\U_{u,T_v}$ if $|v|=h+2$. Then we get that $\pi^{(h+2)}$ the invariant measure of $\U^{(h+2)}$ on the finite tree $T^{(h+2)}$ satisfies, for all $u$ such that $|u|\leq h+1$ the equation
	\[\pi^{(h+2)} \U^{(h+2)}=\pi^{(h+2)}\]
    and by uniqueness $\pi^{(h+2)}$, it is the $h$-invariant measure of $\U^{(h+2)}$ given in \eref{eq:pi}.
	For $u$ such that $|u|\leq h+1$, the equation 
	\ben\label{eq:dq2}
	\pi^{(h+2)}_u=\sum_v \pi^{(h+2)}_v \U^{(h+2)}_{v,u}
	\een 
	can be represented by the formula of $\pi$ instead, on $T^{(h+2)}$; indeed, a quick inspection shows that since the weights involved in the computation of $\pi(u)$ using $\U$ and those for the computation of $\pi^{(h+2)}_u$ using $\U^{(h+2)}$ are the same (these are the transitions on the path on $\cro{\root,u}$ and the weight toward the subtree $f_v$ rooted on $\cro{\root,u}$ thanks to the fact that we took the total $\U^{(h+2)}_{u,v}=\U_{u,T_v}$ if $|v|=h+2$).
	
	From this, we deduce that $\pi_u=\sum_v \pi_v \U_{v,u}$ for all $u$ such that $|u|\leq h$, and since $h$ can be taken as large as wanted, for all $u\in T$.
\end{proof}

\subsection{Additional recurrence criteria: Trees with a finite/infinite number of ends}
\label{sec:qdgg}
It may seem plausible that a Markov chain on  an infinite $T$ with a given irreducible \AUD transition matrix $\U$ satisfies is transient if and only if there exists an end on $T$ that is transient. However, the restriction/projection of a Markov chain on a subspace is not always a Markov chain, so that some work is needed to give a meaning to this intuition, and to observe, that this ``if and only if'' is somehow valid when the tree has a finite number of ends; however, it may fail when the number of ends is infinite. The aim of this section is to shed some lights on the principles into play.  

\subsubsection{Construction of a transition matrix per end.}
\label{sec:CMKE} 

Let $\U$ be an irreducible \AUD transition matrix on an infinite tree $T$, and let ${\sf Ends}$ be its set of ends; this set can be finite, infinite countable, or infinite uncountable. 

Recall that an end $\p$ is an injective path in $T$, starting at $\root$. Two different ends $\p$ and $\q$ share a finite path $\cro{\root,L(\p,\q)}$, adjacent to the root, and then they separate forever. The node $L(\p,\q)$, will be called in the sequel the last common node of $\p$ and $\q$.

We associate with each end $\p$, a subtree $T^\p$ of $T$, consisting of the path $\p$ together with the set of \underbar{finite} subtrees rooted at the neighbors of $\p$.
Hence, for $\p\neq \q$, the intersection of $T^\p$ and $T^\q$ is made of $\cro{\root, L(\p,\q)}$ together with the finite subtrees rooted at the neighbors of this segment.

Let us now define a projection $\wp_\p:T\to T^\p$ as follows:\\
-- if $v\in T_\p$ then set $\wp_\p(v)=v$.\\
-- if $v\notin T_\p$ then $\wp_\p(v)$ is the node of $\p$ that is the closest to $v$ (the highest ancestor of $v$ which is in $\p$).

We will now define a  transition matrix $\U^{\p}$ on $T^\p$.
Intuitively, seeing the pair $(T,\U)$ as weighted graph (where $u\to v$ in $T$ is weighted $\U_{u,v}$), the  transition matrix $\U^\p$ is obtained by redirecting the edges $u\to v$ starting at $u\in T^\p$ and going to $v$, where $v$ is outside $T^\p$, to $\wp_\p(v)$. Formally, for any two nodes $u,v$ in $T^\p$,
\[U^{\p}_{u,v} := \sum_{w \in T} \U_{u,w} \1_{\wp_\p(w)=v}.\]
Observe that $\U^{\p}$ is always an irreducible \AUD  transition matrix (in particular, it is not defective).

\begin{rem}\label{rem:passage} If $\U$ is recurrent and $X=(X_i,i\geq 0)$ is a $\U$-Markov chain, then $\U^\p$ can be seen to be the  transition matrix of the Markov chain $(X_{t_k},k\geq 0)$ where $(t_k)$ is the sequence of passage times on $T^\p$. It is immediate that in this case $\U^{\p}$ is also recurrent; when $\U$ is not recurrent, then $(X_k)$  may visit only a finite number of times $T^\p$: in this case, the process  $(X_{t_k},k\geq 0)$ is not a Markov chain, its  transition matrix is defective (while $\U^\p$ is not defective).
  \end{rem}

  \begin{lem}\label{lem:qfdsf} Assume that $\U$ is an irreducible \AUD transition matrix on an infinite tree $T$. If there exists an end $\p\in T$ such that $\U^{\p}$ is transient, then $\U$ is transient.
  \end{lem} 
  \begin{proof} By contradiction, assume that $\U$ is recurrent and $\U^\p$ is transient on an injective path $\p$. Let $(X_k,k\geq 0)$ be a $\U$ Markov chain. By \Cref{rem:passage}, $\U^\p$ coincides with the  transition matrix of the Markov chain $(X_{t_k},k\geq 0)$, obtained by restricting of $(X_k)$ to $T^\p$, and clearly $(X_{t_k})$ is recurrent on $T^\p$ (since $(X_k)$ visits each node of $T$ infinitely often, its restriction too).
    \end{proof}
   See \Cref{ex:3} for an explicit example.

  \begin{rem}\label{rem:gehtjryg}The reciprocal of \Cref{lem:qfdsf} is wrong: there are examples of \AUD $\U$ on infinite trees that are transient, while the projection $\U^\p$ on all ends are (positive) recurrent. For example, in the complete binary tree, assume that $\U_{u,p(u)}=9/23$, $\U_{u,u1}=\U_{u,u0}=7/23$ for all $u\neq \root$ (for $u=\root$, $p(u)$ does not exist, take $\U_{\root,\root}=9/23$ instead); let $\p$ be the (leftmost) injective path $\p=(\root, 0^1,0^2,0^3,\cdots)$, where $0^j$ stands for the word made with $j$ zeros. We have in this case $T^\p=\p$, and $\U^\p(0^j,0^j)=7/23$, $\U^\p(0^j, 0^{j-1})=9/23$,   $\U^\p(0^j, 0^{j+1})=7/23$ (but for the root, $\U^{\p}(\root,\root)=9/23$) which is the  transition matrix of a recurrent Markov chain (and even positive recurrent), whereas $\U$ is clearly not recurrent, since the distance to the root process is a transient Markov chain.
    \end{rem}
    
\begin{rem}[On the suppression of finite trees]
  For an infinite tree $T$, let $T^\infty$ be the tree obtained by keeping only the nodes of $T$ having an infinite number of descendants. This tree can also be seen as the union of the ends of $T$. Again, one can define a projection $\wp_\infty:T\to T^\infty$ by letting $\wp_\infty(v)$ being the closest node of $v$ in $T^\infty$. If $\U$ is irreducible and \AUD on $T$, let  $\U^\infty$ be defined by edge redirection, as follows : for $u,v \in T^{\infty}$,
  \[\U^\infty_{u,v}= \sum_{w \in T} \U_{u,w} \1_{\wp_\infty(w)=v}.\]
  -- Proving that $\U$ is recurrent on $T$ if and only if $\U^\infty$ is recurrent on $T^\infty$ is straightforward.\\
  -- For the positive recurrence, however, the picture is more complex : for example, imagine that $T$ contains a single end, that we identify with $\N$ in the sequel, and assume that a finite subtree $T_i$ hangs from each $i\in \N$. It is easy to construct a sequence $(T_i)$, for example $T_i$ being isomorphic to a segment $[0,f(i)]$, for which when one enters in $T_i$, the mean exit time of $T_i$ grows rapidly with $i$. It is then possible to find $\U$ that is not positive recurrent, when $\U^\infty$, in which all $T_i$ have been removed, is positive recurrent.   
\end{rem}

We state another lemma that will let us go further into the analysis of the positive recurrence. 
\begin{lem}\label{lem:central} Let $(X_k,k\geq 0)$ be a Markov chain following an irreducible \AUD transition matrix $\U$ on $T$. Then $\U$ is positive recurrent if and only if there   exist $v\in T$, and a  ball $B=B(v,r)$ with radius $r\geq 0$ for the graph distance in $T$,  such that 
	\[
	\sup_{w \in B} \E_w\l(\tau_B^+\r)<+\infty
	\]
	where $\tau_B^+=\inf\{k >0, X_k \in B\}$ for $X$ a $\U$-Markov chain. 
\end{lem}

\begin{proof}	The direct implication is immediate, since an irreducible and positive recurrent Markov chain satisfies $\E_u(\tau_u^+)<\infty$ for every $u\in T$, so it suffices to choose $r=0$ and any $v$ in $T$.

For the reciprocal, assume that there exists a pair $(v,r)$ such that  $\sup_{w \in B(v,r)} \E_w(\tau_{B(v,r)}^+)<+\infty$. Now, let $v^*$ be the closest  node to the root $\root$ in $B$.
To prove that the chain $(X_k,k\geq 0)$ is positive recurrent it will be enough to prove that $\E_{v^*}(\tau_{v^*}^+)<\infty$, since the chain is irreducible.\par
To prove this, we consider the chain $X^B$ defined as the original process restricted to $B$, that is, for $k\geq 0$, $X^B_k=X_{\nu_k}$, where $\nu_k$ is the $k$th visit time of $X$ in $B$; the chain $X^B$ is a.s. well defined, because all the random variables $\nu_k-\nu_{k-1}$ are a.s. finite, and this is a consequence of the fact that $\E(\tau_k-\tau_{k-1})\leq \sup_{w \in B} \E_w(\tau_B^+)<+\infty$. 
 Now, start a Markov chain  $X^B=(X_i^B,i\geq 0)$ at $v^*$ and let $F=\inf\{i>0, X^B_i = v^\star\}$ be the return time at $v^*$. From what it is said above, we can deduce that 
	\[
		\E_{v^*}(\tau_{v^*}^+) \leq \E(F) \times \sup_{w \in B} \E_w(\tau_B).
      \] 
Since, $X^B$ is an irreducible Markov chain of a finite state space, $\E(F)<+\infty$, which allows to conclude. 
\end{proof}

\subsubsection{Finite number of ends}

Let $T$ be a tree with a finite number of ends. 
For each end $\p$ of  $T$, let ${\sf Source}(\p)$ be the vertex of $\p$, the closest to $\root$, such that the subtree hanging from $u$ in $T$ does not intersect any other ends: if $T$ has a single end $\p$, then ${\sf Source}(\p)=\root$, otherwise, the parent of ${\sf Source}(\p)$ in $T$ belongs to several ends. \par
Denote by ${\sf Sources}(T)$, the set of sources of all ends of $T$.
We have
\[|{\sf Sources}(T)|=|{\sf Ends}(T)|.\]
The following theorem completes the picture of \Cref{lem:qfdsf} 
\begin{theo}\label{theo:fnerec} An irreducible \AUD transition matrix $\U$ on an infinite tree $T$ with a finite number of ends is:
  \bir
  \itr recurrent if and only if for all end $\p$ of $T$, the transition matrix $\U^{\p}$ is recurrent on $T^\p$,
  \itr positive recurrent if and only if for all end $\p$ of $T$, the transition matrix $\U^{\p}$ is positive recurrent on $T^\p$.
  \eir
\end{theo}
If the number of ends is infinite, the conclusion of the theorem may fail (see  \Cref{rem:gehtjryg} for an example of tree where all the $\U^p$ are positive recurrent, while $\U$ is transient).
\begin{proof}
\textbf{Proof of $(i)$}.
  The implication is explained in \Cref{rem:passage}: when $\U$ is recurrent, if one extracts from a $U$-Markov chain $(X_k,k\geq 0)$ the process $(X_{t_k},k\geq 0)$ corresponding to the successive passages in $T^\p$, then one obtains a $\U^\p$ Markov chain on $T^\p$, which is necessarily recurrent (since each node of $T^\p$ is visited infinitely often).
    
Let us establish the reciprocal, and for this assume that all the transition matrices $\U^\p$ are recurrent.

We call  ${\sf Core}(T)$ the finite subtree of $T$ obtained 
as the union of $C(T):=\cup_{\p \in {\sf Ends}(T)} \cro{\root, {\sf Source}(\p)}$ together with the finite trees of $T$ hanging from the neighbors of $C(T)$.

Now consider a $\U$-Markov chain $(Y_k,k\geq 0)$ on $T$, starting from any node $u\in T$, and let us prove that $\inf\{j\geq 1: Y_j=\root\}$ is almost surely finite, which is sufficient to conclude by irreducibility.

Assume that for some $k$, $Y_k$ is not in ${\sf Core}(T)$: we will prove that a.s., $(Y_{k+i},i\geq 0)$ will eventually come back in ${\sf Core}(T)$. Since 
 $Y_k\not\in {\sf Core}(T)$, there exists a unique end $\p$, such that $Y_k$ is in $T^\p$; in other words, $Y_k$ is a strict descendant of ${\sf Source}(\p)$. (Recall that $T_{{\sf Source}(\p)}$ is the notation for the tree containing the set of descendants of ${\sf Source}(\p)$ in $T$).

By construction, $\U^{\p}_{u,v}$ coincides with $\U_{u,v}$ when $u$ is in $T_{{\sf Source}(\p)}$. So as long as $Y_{k+i}$ for $i\geq 0$ is in $T_{{\sf Source}(\p)}$, the process  $(Y_{k+i},i\geq 0)$ can be seen as a $\U^\p$ Markov chain. Since this chain is recurrent, it will visit the node ${\sf Source}(\p)$ with probability 1.
Since  ${\sf Source}(\p)$ belongs to ${\sf Core}(T)$, this property implies that the Markov chain $(Y_k,k\geq 0)$ will visit a.s. infinitely often the finite set ${\sf Core}(T)$, which implies by a routine argument, that it visits all of the nodes of the finite set  ${\sf Core}(T)$ infinitely often, including the root $\root$. This concludes the proof of $(i)$.~\medskip

\noindent \textbf{Proof of $(ii)$} As explained in Remark \ref{rem:passage}, when $\U$ is recurrent, $\U^\p$ is the transition matrix of a $\U$-Markov chain observed uniquely at its passage time on $T^\p$. 
Hence for the direct implication, when moreover $\U$ is positive recurrent, then, for all end $\p$,
	\[
	\E_\root^{\p}(\tau_\root^+)\leq \E_\root(\tau_\root^+)<\infty, 
	\]
    where we have decorated with a $\p$ the sign $\E$, and will do the same for $\P$, in order to express that the computation are done with $\U^\p$ instead of $\U$.

	For the reciprocal, we will apply Lemma \ref{lem:central}. We start considering $r$ big enough such that $B=B(\root,r)$ contains all the sources. 
	Since for all $\p$, $\U^{\p}$ is irreducible and the root $\root$ is positive recurrent (for $\U^{\p}$, by hypothesis), one has that
	\[
	\E_{u}^{\p}(\tau_u^+)<\infty\quad \forall u\in T^\p \implies  \E_{u}^{\p}(\tau_B^+)<\infty\quad \forall u\in T^\p.
	\]
	Now we use this to study the expected first return time to $B$ when starting at $w \in B$ under the transition matrix $\U$. When conditioning on the first step of the chain on $T$ one gets 
      \begin{align}\label{eq:pi2}
		\E_w(\tau_B^+) = 1 + \sum_{u\in T\setminus B} \U_{w,u}\E_u(\tau_B^+)\quad \forall w \in B,
	\end{align}	where the 1 comes from the price of the first step, and the rest, decomposes the remaining cost in function of the first step.
      Notice that every $u\in T\setminus B$ belongs to a unique tree $T^\p$ (for a single end $\p$) which implies that  for each $u\in T^\p\setminus B$, 
	\[
	\E_u(\tau_B^+) = \E_u^{\p}(\tau_B^+)<\infty,
	\]
	since the transition probabilities in $T$ coincide with those of $T^\p$ outside $B$. 
	
	From \ref{eq:pi2} we deduce that it will be enough to prove that  $\sum_{u\in T\setminus B} \U_{w,u}\E_u^{\p}(\tau_B^+)<\infty$ for every $w\in B\cap T^\p$ to conclude that $\E_w(\tau_B^+)<\infty$. Hence
    \[
		\sum_{u\in T\setminus B} \U_{w,u}\E_u(\tau_B^+) = \sum_{\p \in{\sf Ends}(T)} \sum_{u\in T^\p\setminus B} \U_{w,u}\E_u^{\p}(\tau_B^+).\]
The number of ends being finite, it remains to prove that $\sum_{u\in T^\p\setminus B} \U_{w,u}\E_u^{\p}(\tau_B^+) <+\infty$ for every $w\in T^\p\cap B$. Now we know that for all end $\p$,
	\begin{align}\label{eq:pi3}
		\infty>\E_w^{\p}(\tau_B^+) = 1 + \sum_{u\in T^\p\setminus B} \U_{w,u}\E_u^{\p}(\tau_B^+)\quad \forall w \in B \cap T^\p.
	\end{align}
	From this we deduce that $\E_w^T(\tau_B^+)<\infty$.
	We finish by noticing that the ball $B$ has a finite number of points, we obtain the hypothesis of Lemma \ref{lem:central} which let us conclude that the chain in $T$ is positive recurrent.
\end{proof}
 
\subsubsection{Infinite number of ends}

When the number of ends is infinite, the main tools of the previous section leading to Theorem \ref{theo:fnerec} disappear. Remark \ref{rem:gehtjryg} gives an example of tree with an infinite number of ends, for which each end $\p$ has its associated transition matrix $\U^\p$ positive recurrent, while the global $\U$-Markov chain is transient.

When the number of ends is infinite, we are left with \Cref{theo:1.3}, which provides the \textsf{h}-invariant measure $\pi$ of $\U$, and then characterizes the positive recurrence. Proposition \ref{pro:rec} provides a general characterization of recurrence.

\section{\AUD transition matrices : more properties}
\label{sec:aud}

The next section provides an iterative way to compute the $h$-invariant measure $\pi$ on a finite or infinite tree $T$, by computing $\pi_u$ node by node, using only simple computations (which avoid the determinantal formulas). 

\subsection{On a leaf addition strategy}

We first provide the general principle:
\begin{pro} \label{pro:dadgty}
Let $\tau$ and $\tau'$ be two finite or infinite trees such that  $\tau'$ can be obtained from $\tau$ by the suppression of a leaf $\ell$, that is $\tau'=\tau \setminus\{\ell\}$.\par 
Assume that $\U$ and $\U'$ are two \AUD transition matrices on $\tau$ and $\tau'$, respectively, such that:\\
\bls if $u\in \tau'$ and $v\in \tau'\setminus\{p(\ell)\}$ then  $\U_{u,v}'=\U_{u,v}$,\\
\bls if $u\in \tau'$ then  $\U_{u,p(\ell)}'=\U_{u,p(\ell)}+\U_{u,\ell}$ ( the transitions in $\U$ toward $\ell$ are redirected toward $p(\ell)$ in $\U'$).\par
Then,\bir
\itr if $\rho'=(\rho'_u,u\in \tau')$ is a 1-left eigenvector of $\;\U'$, then   $\rho$  defined by
 \ben\label{eq:gtey}
\bpar{ccl}\rho_u&=& \rho'_u,~~\textrm{if }u\neq \ell,\\
\rho_\ell&=& \dis\sum_{v\in \cro{\root,p(\ell)}} \rho'_v U_{v,\ell } /\U_{\ell,p(\ell)}\epar \een
is a 1-left eigenvector of $\U$.
\itr If $\rho'$ is the $h$-invariant measure $\pi'$ of $\;\U'$, then $\rho$ defined in \eref{eq:gtey} is the  $h$-invariant measure $\pi$ of $\;\U$.
\eir 
\end{pro}

\begin{rem}
  The previous result is a consequence of the ergodic theorem if $\tau$ is finite or $\;\U'$ positive recurrent. In these cases, there is a single invariant measure which is proportional to the a.s. asymptotic  fraction of time spent at each vertex. Now, $\U'$ is the transition matrix of a $\U$-Markov chain $(X_i,i\geq 0)$ observed at its passage time $(t_k,k\geq 0)$ in $\tau'$.  Using the ``normalization'' $\rho_\root=\rho'_\root=1$, we get equality of $\rho_u=\rho_u'$ for $u\neq \ell$ and from here, the second formula in \eref{eq:gtey} is the balance equation at $\ell$ (preservation of $\rho_\ell$ by $\U$).
  
  This argument can be turned into a complete proof in general trees, by working level by level as in the proof of \Cref{theo:1.3}, but will instead provide an algebraic argument below.
  \end{rem}
  \begin{rem}  
    This proposition implies that the projection $\U\to \U'$ can only reduce the dimension of the eigenspace of the eigenvalue 1.
\end{rem}
 
\begin{proof}[Proof of \Cref{pro:dadgty}]
  $(i)$ We only need to check that $\rho$ as defined in \eref{eq:gtey} satisfies for all $u\in \tau$,
  \ben\label{eq:qfHKgt}
  \rho_u = \sum_{v\in \cro{\root,u}}\rho_v \U_{v,u}+\sum_{v \in c_\tau(u)} \rho_v \U_{v,u}.
  \een
  For a given node $u$, this equation relates only the values  $\rho_w$ for $w$ in $\cro{\root,u}\cup c_\tau(u)$. Hence since we kept $\rho_w=\rho_w'$ for all $w\neq \ell$, and all the $\U_{u,v}'=\U_{u,v}$ for $v\neq p(\ell)$, we see that \eref{eq:qfHKgt} holds for all nodes $u \notin\{\ell,p(\ell)\}$ since $\rho'=\rho' \U'$.

  Formula \eref{eq:gtey} is designed so that $\rho_\ell$ satisfies the balance equation \eref{eq:qfHKgt} for $\U$, if $\rho'$ and $\rho$ coincide on $\tau'$. It remains to show that $\rho_{p(\ell)}=\rho'_{p(\ell)}$ satisfies also \eref{eq:qfHKgt} (which logically, consists in checking that \eref{eq:qfHKgt} is consistent with \eref{eq:gtey}, at $u=\ell$ and $u=p(\ell)$).
 
Write the balance equation for $\rho'_{p(\ell)}$ under $U'$: 
\ben\label{eq:rghtehzz} \rho'_{p(\ell)}= \sum_{v\in\cro{\root,p(\ell)}} \rho'_{v}\U'_{v,p(\ell)}+\sum_{v \in c_{\tau'}(p(\ell))} \rho'_v \U'_{v,p(\ell)}.\een
Replace in this formula  $c_{\tau'}(p(\ell))=c_{\tau}(p(\ell)) \backslash\{\ell\}$, and  $\U_{u,p(\ell)}'=\U_{u,p(\ell)}+\U_{u,\ell}$, and use \eref{eq:gtey} to get
\ben\label{eq:rghtehzz2} \rho'_{p(\ell)}&=& \sum_{v\in\cro{\root,p(\ell)}} \rho'_{v}\l(\U_{v,p(\ell)}+\U_{v,\ell}\r) +\sum_{v \in c_{\tau}(p(\ell)) \backslash\{\ell\}} \rho'_v \U'_{v,p(\ell)}.\\
&=& \sum_{v\in\cro{\root,p(\ell)}} \rho'_{v}\U_{v,p(\ell)}+\rho_\ell \U_{\ell,p(\ell)} +\sum_{v \in c_{\tau}(p(\ell)) \backslash\{\ell\}} \rho_v \U_{v,p(\ell)}
\een
Here we see that by taking $\rho_{p(\ell)}=\rho'_{p(\ell)}$ this equation is exactly the balance equation \eref{eq:qfHKgt} for $u=p(\ell)$.\\
$(ii)$ On finite trees, the statement is clear, by the uniqueness of the invariant measure. On infinite trees, since the computation of $\rho_\ell$ by \eref{eq:qfHKgt} and by the determinantal formula depends only on the branch $\cro{\root,\ell}$ (and the projection of the transition matrix $\U$ onto this branch using $\uU$ for $u=\ell$), it suffices to observe that any projection of $\U$ on a finite branch can be realized by a finite transition matrices (obtained, by cutting $T$ at the level $|\ell|+1$ for example, and making some weight adjustments). Hence, \eref{eq:qfHKgt} and the determinantal formulas must coincide at $\ell$. 
\end{proof}

\subsection{On the iteration of the leaf addition strategy  to compute the $\textsf{h}$-invariant distribution}\label{sec:leaf}

The leaf addition strategy allows us to compute the  $\textsf{h}$-invariant measure $\pi$  ``node by node'' using a very simple algorithm, in which the formula \eref{eq:pi} plays no role.  Given $\U$ an \AUD irreducible transition matrix on $T$, it will suffice to grow $T$ node by node to achieve this goal. Before doing this, it is useful to introduce the notion of ``projected transition matrix'', which provides the right structural decomposition:

\paragraph{Projected transition matrix.}
Let $\U$ be an irreducible \AUD transition matrix on a finite or infinite tree $T$.
If $t$ is a subtree of $T$ (containing $\root$), let us define a transition matrix $\U^{t}$ on $t$ that we will call the projection of $\U$ on $t$, by redirecting  all the transitions toward the exterior of $t$, towards the nearest node of $t$, as follows:
\ben\label{eq:proj}\U_{u,v}^{t}= \U_{u,T_v} - \sum_{w \in c_{t}(v)} \U_{u,T_w}, \textrm{ for }u,v\in t.\een
Take $\U$ an \AUD irreducible transition matrix on a finite or infinite tree $T$, and a sequence of trees $T{(k)}$ such that $T(k)$ is obtained from $T(k-1)$ by the addition of a leaf $u_k$, and such that $\lim_{n\to\infty}T(n)= \cup_k T(k)= T$ (for example, $T(k)$ defined as the set of the $k$ smallest nodes of $T$ for the breadth first order would do).

  To the sequence $(T(k))$ is then associated a sequence of projected transition matrices $(\U^{T(k)})$ obtained by projecting $\U$ on $T(k)$. By
  \eref{eq:proj},
 \ben\label{eq:proj2}\U_{u,v}^{T(k)}= \U_{u,T_v} - \sum_{w \in c_{t}(v)} \U_{u,T_w}, \textrm{ for }u,v\in T(k).\een
Now, we claim that for each $k$, the pair $(\U^{T(k-1)},\U^{T{(k)}})$ satisfies the hypothesis for $(\U',\U)$ in  \Cref{pro:dadgty}.
Indeed, since $u_k$ is an added leaf in $T(k-1)$, if $v\neq \{u_k,p(u_k)\}$ and $u\in T(k-1)$,
\[ \U^{T(k-1)}_{u,v}= \U^{T(k)}_{u,v},\]
and for $v= p(u_k)$,
\ben\label{eq:tkm1} \U^{T(k-1)}_{u,p(u_k)}= \U_{u,T_{p(u_k)}} - \sum_{w \in c_{T(k-1)}(p_{u_k})} \U_{u,T_w}\een
Since $u_k$ is a leaf of $T(k)$, and using that  $c_{T(k)}(p_{u_k})$ is made of $u_k$ and of $c_{T(k-1)}(p_{u_k})$,
\ben\label{eq:tk} \U^{T(k)}_{u,u_k}= \U_{u,T_{u_k}},\een and
\be \U^{T(k)}_{u,p(u_k)}&=&\U_{u,T_{p(u_k)}}-\sum_{w\in c_{T(k)}(p(u_k))} \U_{u,T_w}\\
&=&\U_{u,T_{p(u_k)}}-\U_{u,T_{u_k}}-\sum_{w\in c_{T(k-1)}(p(u_k))} \U_{u,T_w}.   \ee
Therefore

by \eref{eq:tk} and \eref{eq:tkm1}, this gives indeed
\ben\U^{T(k-1)}_{u,p(u_k)}=\U^{T(k)}_{u,p(u_k)}+ \U^{T(k)}_{u,u_k}.\een
A simple iteration of \Cref{pro:dadgty}, together with Kolmogorov extension theorem give immediately:
\begin{theo}\label{theo:lgp}Let $\pi^{T{(k)}}$ be  the ${\sf h}$-invariant measure of $\;\U^{T{(k)}}$ on the finite tree $T{(k)}$, normalized by $\pi_{\root}^{T{(k)}}=1$.
	\bir
	\itr 	$(a)$ the measures  $\pi^{T{(k)}}$ and $\pi^{T{(k-1)}}$ coincide on $T{(k-1)}$,\\
	$(b)$ to compute the value $\pi^{T{(k)}}_{u_k}$ (on the additional node) using $\pi^{T{(k-1)}}$,
	\ben \label{eq:fgeht}
	\pi^{T{(k)}}_{u_k}    ={\dis\sum_{v\in \cro{\root,p(u_k)}} \pi^{T{(k-1)}}_{v} \U_{v,T_{u_k}}}/ {\U_{u_k,p(u_k)}} .\een
	\itr The measures $(\pi^{T{(k)}},k\geq 0)$ are consistent (that is $\pi^{T{(k+1)}}$ and  $\pi^{T{(k)}}$ coincide on $T{(k)}$, for all $k$), so that there exists a measure $\rho$ on $T$, such that $\rho\big|_{T{(k)}}=\pi^{T{(k)}}$. This measure $\rho$ is the \textsf{h}-invariant measure $\pi$ of $\U$.
	\eir
\end{theo}

\begin{rem}\label{rem_lat}
 Formula \eref{eq:fgeht} provides a way to compute the ${\sf h}$-invariant distribution of a finite tree quite efficiently starting from an initial tree reduced to its root, and by the successive addition of edges (see , Examples \ref{ex:2},\ref{exa:FC},\ref{Ex_line},\ref{exa:d-ary}).  The computation of $\pi_u$ requires $|u|$ multiplications, the sum of $|u|$ terms, and one division, so that, on a tree with $n$ nodes, for $H:=\sum_{u\in T} |u|$, $H$ sums, $H$ multiplications and $n$ divisions; since in all trees $H\leq n^2/2$, this is generally less expensive than the formula of Theorem \ref{theo:1.3}.
    \end{rem}

  \begin{cor}\label{cor:las} Take two (finite or infinite) trees $T^{(1)}$ and $T^{(2)}$, each of them being respectively equipped with an irreducible \AUD transition matrix $\U_{(1)}$ and $\U_{(2)}$. Denote by $\pi^{(1)}$ and $\pi^{(2)}$ their respective $h$-invariant measures. Assume that both $T^{(1)}$ and $T^{(2)}$ contain the same node $u$, so that both trees contain the tree $\tau$ reduced to $\cro{\root,u}$.\\
    If the two projected transition matrices $\U_{(1)}^{\tau}$ and  $\U_{(2)}^{\tau}$ are equal, then  $\pi^{(1)}(u) = \pi^{(2)}(u)$.  
\end{cor}
This is a direct consequence of the leaf addition strategy, since one can grow $\tau$ within $T^{(1)}$ and $T^{(2)}$ by successively adding the same temporary ``leaf'' between the ancestors of $u$ at levels $h$ and $h+1$, for $h$ going from 0 to $|u|-1$.

\color{black}
 
\subsubsection{Left Eigenspace of an \AUD transition matrix, for $\lambda=1$} \label{sec:LECTM}

We will mainly explain here how to compute, or characterize, other left eigenvector(s) of $\U$ associated with eigenvalue 1; the extraction of invariant measures can be seen as the filtering out those with positive coordinates (which is difficult in practice, when possible). 

If $\U$ is an irreducible transition matrix on a finite tree, by the Perron-Frobeniüs theorem, the left eigenspace associated to the eigenvalue 1 has dimension 1, and by a consequence of \Cref{theo:1.3}, it is  generated by  the \textsf{h}-invariant measure $\pi$.
When $T$ is infinite, the picture is completely different, since this left eigenspace may be high dimensional, and even the cone of positive invariant measures can generate a space of infinite dimension (an example is given in \eref{eq:qsgg}). 
Take a tree $T$ with at least one end, and consider $P(T)$ the subtree  \footnote{Technically $P(T)$ is rather a subset of $T$, since it does not satisfy the $(iii)$ of the definition of trees given at the beginning of Section \ref{sec:TF}. It is a tree up to a relabeling of the nodes.} of $T$, made up of the union of the infinite injective paths starting at $\root$ in $T$. Since $P(T)$ has no leaf, each node $u$ in $P(T)$ has at least one child:  $\l|c_{P(T)}(u)\r|\geq 1$.
Set
  \ben\label{eq:QT} Q(T):= 1+\sum_{u \in P(T)} \l(|c_{P(T)}(u)|-1\r).\een

\begin{pro}\label{pro:qff}The 1-left eigenspace of an irreducible \AUD transition matrix $\U$ has dimension $Q(T)$. Hence, the vector space generated by the set of invariant measures, has dimension $d$ with $1\leq d \leq Q(T)$.
\end{pro}
An example of transition matrix with infinite dimensional 1-left eigenspace is presented in Example \ref{sec:qfgrht}.
\begin{proof} 
 A  row vector $L=\bma L_{u}, u \in T\ema$ is a  1-left eigenvector of $\U$ iff it solves the infinite system of equations:
 \ben\label{eq:LeftE} L_{u}(1-\U_{u,u})= {\sum_{v \in {c_T(u)}} L_v\U_{v,u} + \sum_{v\in\cro{\root,p(u)}} L_v \U_{v,u}}, ~~~ u \in T.
 \een
 One can try to solve this system progressively: typically, one would force a value $L_\root$ at the root (only $L_\root=1$ and $L_\root=0$ need to be considered), and from there, one could search for some values $L_v$ for all the other nodes, so that the complete system is satisfied.

Assume that for a subset of nodes $A$ of $T$, some values $(L_u,u \in A)$ have been fixed: we say that $(L_u,u \in A)$ is a pre-solution, if for all $u$ such that  $\cro{\root,u}\cup\l\{ ui, i\leq |C_{P(T)}(u)|\r\}$ is contained in $A$, then \eref{eq:LeftE} is satisfied: in other words, we only require that the equations \eref{eq:LeftE} involving only elements of $A$ are satisfied. A pre-solution may or may not be extendable to the complete tree; however, any restrictions of a complete solution of $T$ on a subset $A$, is a pre-solution.

Assume that we have in hand a pre-solution $(L_u,u\in A)$ for a fixed subset $A$ of $T$, path connected to $\root$. If $A$ has a node $w$ such that one of its children $v$ is not in $A$, and such that $T_v$ is finite, we claim that there exists a unique way to prolong the pre-solution of $A$ to $A\cup T_v$ (such that this prolonged sequence is a pre-solution on $A\cup T_v$)~:
 this is a consequence of the equation
 \[L_{\cro{\root,p(v)}}M_{\cro{\root,p(v)},T_v}+L_{T_v}M_{T_v,T_v}=L_{T_v}\]
 where $L_{\cro{\root,p(v)}}$ denote the restriction of $L$ to the sequence ${\cro{\root,p(v)}}$, and more generally, for two subsets $A$ and $B$ of indices $M_{A,B}$ is the matrices obtained by conserving the rows of $M$ indexed by the elements of $A$ and the columns indexed by the elements of $B$. 
Since $M_{T_v,T_v}$ does not have eigenvalue 1 (since $M_{T_v,T_v}$ is a defective  transition matrix, its largest eigenvalue in modulus is less than 1),  $M_{T_v,T_v}-\Id_{T_v}$ is invertible, so that the desired extension is given by $L_{T_v}=L_{\cro{\root,p(v)}}M_{\cro{\root,p(v)},T_v}(\Id-M_{T_v,T_v})^{-1}$.
  
Since the 1-left eigenspace is generated by the pre-solution that can be extended to the tree $T$, we may from now on, remove all the finite subtrees hanging from the injective infinite paths starting at $\root$, and prove the proposition for trees without leaves.

Observe that if one takes a tree $T$ reduced to a single path,
then any $\U$ on $T$ possesses a unique one-dimensional 1-left eigenspace, since fixing $L_\root=1$, can be extended in a unique way by \eref{eq:LeftE} (and $L_\root=0$, by extension, leads to the zero eigenvector).

Now, start again to construct a pre-solution, by starting on $A=\{\root\}$, and by fixing $L_{\root}=1$ ``the principal pre-solution'', and the alternative pre-solution by setting $L_{\root}=0$ if $\root$ has degree at least 2.

Consider a pre-solution on $A$, and let us take a node $u\in A$ having all its children out of $A$.
We may add all these children into $A$, if one chooses some values $(L_{ui}, 1\leq i \leq |c_T(u)|)$ satisfying \eref{eq:LeftE}: this is just an equation for $|c_T(u)|$ variables, and this gives $|c_T(u)|-1$ parameters. Since all pre-solution can be extended on a infinite tree with no leaves, the results follow.
\end{proof}

\begin{rem}It is convenient to see the appearance of the multiplicity of invariant measures as an emergent effect allowed by ends. When one writes the invariance equations solved by the measure on a finite tree, fixing $\rho_\root=1$, and tries to compute successively the values of $(\rho_{ui},1\leq i \leq c_u(T))$ using $\rho_{u}$ (the weight of the children using that of their parents), it seems that there are a lot of degrees of freedom, since, as in the proof, a linear combination of their sum $\sum_i \rho_{ui}\U_{ui,u}+\sum_{w\in\cro{\root,u}} \rho_w\U_{w,u}=\rho_u$ contributes to $\rho_u$. However, this is just an illusion, because, the Perron-Frobeniüs theorem tells us that there is a single solution (and this is also the technical argument developed above that there is a single way to extend a pre-solution to a finite subtree). In infinite trees, the picture is different, because the ends represent somehow missing leaves, and paths along which one can always solve the balance equation concerning $\rho_u$ by taking a convenient value of $\rho_{ui}$, for $ui$ at the end. This fact will be discussed again in Section \ref{sec:noleaves}. 
\end{rem}

\subsection{The \textsf{h}-measure $\pi$ and a sequence of \textsf{h}-transform Markov chains}
\label{seq:h-t}
The \textsf{h}-measure $\pi=\pi^T$ of an \AUD irreducible transition matrix $\U$ on an infinite tree $T$ has a striking property: by the leaf addition strategy, if $t$ is a finite subtree of $T$ (containing the root), then, the restriction of $\pi^T$ to $t$ corresponds to $\pi^t$ when $t$ is equipped with the projected transition matrix $\U^t$ (defined in \eref{eq:proj}).
  Since $t$ is finite,  $\U^t$ is positive recurrent. Hence, by taking a sequence of growing finite trees $t(n)$ converging to $T$ (in the sense that $\cup t(n)=T$), one sees that $\pi^T$ has the property of coinciding, as far as we wish from $\root$ with the invariant measure of a recurrent Markov chain even when $\U$ is not at all recurrent.

The same phenomenon can be observed, for example, on the symmetric random walk on $\Z$, for which $\pi=(1, i\in \Z)$ is invariant, and this measure is the limiting measure (for the local topology) of the invariant measure $[1/2,1,\cdots,1,1/2]$ of a simple random walk  on the segment $[-n,n]$, reflected at the boundary.\par

Intuitively, the \textsf{h}-measure $\pi^T$ can be seen as the invariant measure of $\U$ ``forced to be recurrent''. The Doob \textsf{h}-transform helps to give a more formal sense to this statement.
Consider again a finite subtree $t$ of $T$. The Doob \textsf{h}-transform helps to condition the Markov chain under the event $\tau_{t}<\infty$ (``to hit $t$ a.s.''). Introduce
\[h_{t}(u) = P_u(\tau_{t}<\infty),~~~\forall u\in T;\]
hence, $h_t(u)=1$ when $u$ is in $t$ and, $h_t(u)>0$ in all generality.
Now, define the transition matrix $\tilde{\U} = \tilde{\U}(t)$ by
\ben
\tilde{\U}_{i,j}(t) & = & \U_{i,j}~~~\textrm{if } i\in  t,\\
\tilde{\U}_{i,j}(t) & = & \l(h_{t}(j)/h_{t}(i)\r)\U_{i,j}~~~\textrm{if } i\not\in  t,
\een
which is the transition matrix of a $\U$-Markov chain, conditioned to come back to $t$.\\ 
$\bullet$  Thanks to the leaf addition strategy (or Corollary \ref{cor:las}), the invariant measure $\tilde{\U}(t)$ must coincide with $\pi^T$  on ``the interior $t^\circ$ of $t$'' (which is the set of nodes having their ancestors and children in $t$)\footnote{It is not true in generally, since the  $1$-left eigenvector $L^{(t)}$ of $\tilde{\U}(t)$ satisfies:
\[L_b^{(t)} = \sum_{a \in T} L_a^{(t)} \U_{a,b}   =   \sum_{a\in t} L^{(t)} \U_{a,b} + \sum_{a\notin t} L^{(t)} U_{a,b} =  \sum_{a\in t} L_a^{(t)} \tilde{\U}_{a,b} + \sum_{a\notin t} L_a^{(t)} \frac{h_t(a)}{h_t(b)} \tilde{\U}_{a,b}\]
and the simplification  $L_b^{(t)} = \sum_{a\in t} L_a^{(t)} \tilde{U}_{a,b}$ is only valid when  $(\cro{\root,b} \cup c_T(b))\subset t$} . But since $\tilde{\U}(t)$ is recurrent, $\tilde{\U}(t)$ has a single invariant measure $\pi_{(t)}$, and this measure must coincide with $\pi^h(\U)$ on $t^{\circ}$.

Hence, $\pi^T$ can be seen as a local limit of these invariant measures $\pi_{(t_n)}$, corresponding to positive recurrent matrices on spaces $t_n$ locally converging to $T$.

\subsubsection{An example of tree with one end having an eigenvalue $\neq 1$ with infinite multiplicity}
\label{eq:qsgg}
 
Consider the tree $\tau$ with only 4 vertices  $\root,0,1,2$, and the following transition matrix 
\[ \U = \left[ \begin {array}{cccc} {\frac{1}{20}}&{\frac{1}{4}}&{\frac{1}{5}
}&{\frac{1}{2}}\\ \noalign{\medskip}{\frac{1}{3}}&{\frac{2}{3}}&0&0
\\ \noalign{\medskip}{\frac{1}{3}}&0&{\frac{2}{3}}&0
\\ \noalign{\medskip}{\frac{1}{3}}&0&0&{\frac{2}{3}}\end {array}
\right]
\]
where the entries are indexed by $\root,0,1,2$ in this order. 

It can easily be checked that the multiset of eigenvalues of this matrix is $\{1,-17/60,2/3,2/3\}$.
The two left-eigenvectors corresponding to the eigenvalue $2/3$ are $(0,-1,1,0)$ and $(0,-1,0,1)$; the main eigenvector (invariant distribution), corresponding to the eigenvalue $1$, is $\bma 20 & 15 & 12 & 30 \ema/77$ and the eigenvector associated with $-17/60$ is $\bma -19 & 5 & 4 & 10\ema$. So the two eigenvectors $v_1=(0,-1,1,0)$ and $v_2=(0,-1,0,1)$ generates the $2/3$-left eigenspace. 

\paragraph{Key observation:}The coordinates of $v_1$ and $v_2$ corresponding to the root are zero: in the ``balance equation'' $(2/3)\rho = \rho\U$, ``nothing is sent from the root'', and ``what arrives at the root cancel out''.   
\begin{figure}[htbp]
	\centerline{\includegraphics[width = 12cm]{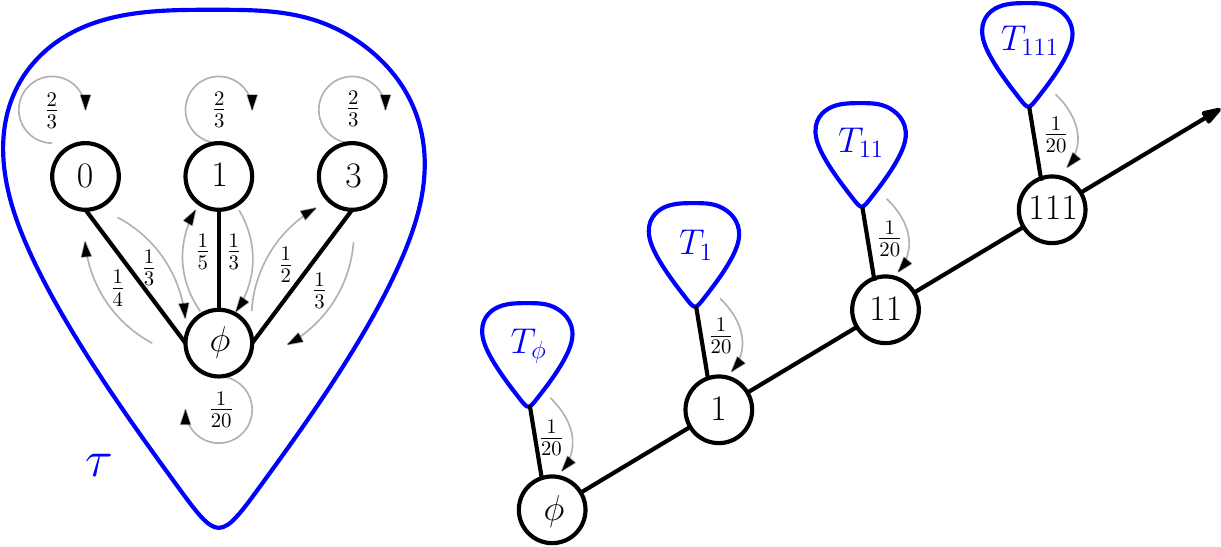}}
	\captionn{	\label{fig:ex}To the right, the infinite tree $T$ which is constructed by grafting copies $(T_u)$ of the tree $\tau$ (to the left of the image) at each vertex of a copy of $\N$.}

    	\end{figure}
    	
Now, let us construct an infinite tree (see Figure \ref{fig:ex}), made of a single line (isomorphic to $\mathbb{N}$) from which hangs, at each node, a copy of $\tau$, as follows.
Consider the infinite tree $T$ composed by:\\
-- the set of nodes on the infinite line $\{1^j,j\geq 0\}$,\\
-- a tree which is a copy of $\tau$ grafted at each node $u$ of type $1^j0$ being $T_u=u \tau$ (in other words, $1^j0$, $1^j00$, $1^j01$, $1^j02$).

Now, let us define a transition matrix $\bar\U$ on $T$:\\
-- for any node $1^j$ in the new infinite branch $\{1^j, j\geq 0\}$, we set  $\bar{\U}_{1^j,.}$ to be the uniform law on the set of neighbors of its neighbors (the neighbors are $1^{j+1}, 1^j0$ and $1^{j-1}$ when $j>0$),\\
-- from the roots of the copies of $\tau$, \\
  $\bullet$  we keep the transitions to the children $\bar{U}_{1^j0, 1^j0k}=U_{\root, k}$, that is $\bar{\U}_{1^j0,1^j00}=1/4$,  $\bar{\U}_{1^j0,1^j01}=1/5$, $\bar{\U}_{1^j0,1^j02}=1/2$, \\
  $\bullet$ but replace the weight $\U_{\root,\root}$ of the loop $\root\to \root$ by a transition to the neighbor on the infinite branch: set $\bar{\U}_{1^j0,1^j0}=0$, $\bar{\U}_{1^j0,1^j}=1/20$.
 
Now, it is simple to check (using the key observation above) that the left eigenspace associated with the eigenvalue $2/3$ is infinite dimensional, since it contains all the vectors $V^{i,1}$ and $V^{i,2}$ (that are clearly linearly independent) where:\\
  -- the row vector  $V^{i,1}=(V^{i,1}_u,u\in T)$ has all its coordinate zero, except
  \[\l( V^{i,1}_{1^i00}, V^{i,1}_{1^i01},  V^{i,1}_{1^i02}\r)=(-1,1,0)\]
  and similarly, for $V^{i,2}$ constructed with the second finite 2/3 left-eigenvector, 
   \[\l( V^{i,2}_{1^i00}, V^{i,2}_{1^i01},  V^{i,2}_{1^i02}\r)=(-1,0,1).\]  

\subsubsection{Combinatorial representation of $\lambda$-left eigenvectors on finite trees}

In the standard computation of the unique invariant probability measure $\rho$ of an irreducible Markov chain with  transition matrix $M$ on a finite state space $E=\{1,\cdots,n\}$, instead of solving
$\mu M=\mu$, one defines $\widetilde{M}$ as the matrix obtained from $M$ by replacing   its last column by $\bma 1 & \cdots 1\ema^t$; one solves instead the morally equivalent system  $ \rho \widetilde{M}= \bma \rho_1 & \cdots &\rho_{n-1}&1\ema$,
or equivalently
\ben\label{eq:solpf} \rho (\widetilde{\Id}-\widetilde{M}) = \bma 0& \cdots &0&-1\ema \een
where $\widetilde{\Id}$ is obtained by replacing the last diagonal coefficient of the identity matrix by zero.
The word ``morally'' is here because \eref{eq:solpf} is an affine system admitting a probability distribution as solution, which generates the one-dimensional vector space, solution of the linear system $\mu M=\mu$.
The non-singular system \eref{eq:solpf} can thereafter be solved by Kramer formula~:
\[ \rho = {\sf Cste}\bma \det( (\widetilde{\Id}-\widetilde{M})_{*i}), 1 \leq i \leq n\ema,\]
where  $(\widetilde{\Id}-\widetilde{M})_{*i}$ is obtained by replacing the $i$th row of $\widetilde{\Id}-\widetilde{M}$ by $\bma 0& \cdots &0&-1\ema$. Up to the sign, it amounts to computing the determinant of the matrix obtained by the suppression  of the $i$th row and last column of $(\widetilde{\Id}-\widetilde{M})_{*i}$: then one gets the same results as if, instead, the $i$th row and last column of ``the initial matrix'' ${\Id} - {M}$ were removed, that is 
\[ \rho = {\sf Cst}.\det((-1)^{n+i} ({\Id} - {M})^{(i,n)}),  \]
where $(i,n)$ is here to denote the suppression of row $i$ and column $n$. But in ${\Id} - {M}$, the sum of the columns is equal to the zero vector: hence, in the determinant computation, the $i$th column can be replaced by the opposite of the sum of the others; by matrix manipulation (adding columns already present to column $i$ does not change the determinant) we get $\rho= {\sf Cst}. \det(( {\Id} - {M})^{(i,i)})  $. After that the matrix tree theorem, provides an expansion in terms of weighted spanning trees (see e.g. Zeilberger \cite{DZ} and references therein, and for a probabilistic point of views, Aldous \cite{Al90}, Broder \cite{Bro89},  Hu, Lyons \& Tang \cite{MR4260489} and, by the authors \cite{LFJFM}).
 
When $\lambda\neq 1$, as explained above, the dimension of the $\lambda$-left eigenspace may have dimension larger than one, so that it cannot be simple, even if, as taught in linear algebra classes, the left-kernel of $\lambda \Id-\U$ can be computed, for example, by performing a triangulation; here we would like to give a formula, having a close algebraic form, valid when $\lambda$ has multiplicity 1.
 
 The above approach cannot be applied directly, because $\lambda{\Id} - {M}$ is no longer a Laplacian matrix: the sum of its columns is never, for $\lambda\neq 1$, the null vector.
 However, we know that $\{\mu~: \mu(\lambda{\Id} - {M})  = 0\}$ is a one dimensional vector space. Consider the matrix $\overline{\lambda{\Id} - {M}}$ obtained by replacing the last column of $\lambda{\Id} - {M}$, by the opposite of the sum of the others (so that now, this matrix is a Laplacian matrix).
The last column of $\overline{\lambda{\Id} - {M}}$ is then \[{\sf LC}:=\l( -\sum_{j=1}^{n-1} \lambda \1_{i=j} - \l(M_{i,j}\r)\r) _{i}.\]

Instead of solving $\mu(\lambda\Id-M)  = 0$, we will solve  the equivalent problem $\nu\l(\overline{\lambda\Id-M}\r) =\bma 0 & \cdots &0\ema$. The problem is equivalent because, if we take a solution $\nu$ to $\nu M =\lambda\nu$,
  then the coordinates of $\nu$ sums up to zero, and since  $\sum_{i\leq n, j\leq n-1}\lambda \nu_{i}\1_{i=j}=\sum_{j\leq n-1}\nu_j=-\lambda \nu_n$,  and  $\sum_{i\leq n, j\leq n-1}\nu_{i}M_{i,j}=\sum_{i\leq n}\nu_i(1-M_{i,n})=-\sum_{i\leq n}\nu_iM_{i,n}=-\lambda \nu_n$, and then $\nu {\sf LC}=0$, so that $\nu\overline{\lambda\Id-M}=0$ (the reciprocal can be shown in the same way, and $\lambda$ having multiplicity 1 allows to conclude).
  From here, since $\overline{\lambda\Id-M}$ is Laplacian, we can adapt the method explained above of the $\lambda=1$ case, and will replace again the last column of  $\overline{\lambda\Id-M}$. The difficulty here is that we don't know any linear combination of coordinates of $\nu$ that is non-zero: it is possible to construct trees, for which for an eigenvalue $\lambda\neq 1$, $\nu$ is non-zero on only two coordinates (corresponding to 2 leaves with same parent).\\
  However, we know that at least one coordinate of $\nu$ is non-zero, and it is then possible to add a condition as $\nu_1=1$, and if no solution is found, add $\nu_2=1$ instead,... and so on, and so forth, until a solution is found with this additional condition. Let us assume, for the sequel, that there exists a solution with $\nu_n=1$.
  We therefore search instead of $\nu\overline{\lambda\Id-M}=0$ a solution to $\nu {M^\triangle}=\bma 0&\cdots & 0 & 1\ema$ where ${M^\triangle}$ has its first $n-1$ columns equal to those of $\overline{\lambda\Id-M}$, and the last one is the canonical vector $e_n$. One can solve $\nu {M^\triangle}=\bma 0&\cdots & 0 & 1\ema$ by Kramer's formula, and get
 \be
 \nu &=& {\sf Cst}. \bma \det(M^{\triangle}_{*i}), 1\leq i \leq n \ema
\ee with again $M^{\triangle}_{* i}$ amounts to replacing the $i$th row of $M^{\triangle}$ by $\bma 0& \cdots &0&1\ema$. But again, by expansion of the determinant on the $i$th line, it is visible that this is equal also to  
 \be
 \nu &=& {\sf Cst}. \bma \det(\overbar{\lambda\Id-M})_{*i}, 1\leq i \leq n \ema\\
 &=& {\sf Cst}.\bma \det((-1)^{n+i} (\overbar{\lambda\Id-M})^{(i,n)}) , 1\leq i \leq n \ema
 \ee
 and one concludes again that since $(\overbar{\lambda\Id-M})$ is Laplacian, one can still play with the columns to get:
 \[\nu =  {\sf Cst}.\bma  \det( (\overbar{\lambda\Id-M})^{(i,i)}), 1\leq i \leq n \ema.\]
From here we can make an expansion on the set of spanning trees (on which the edges are weighted by $\tilde{M}$).

\subsection{Eigenspaces of Markov chains on trees with  no leaf}\label{sec:noleaves}

\label{sec:eig}
In this section, we discuss the properties of the spectrum of \AUD transition matrices on infinite trees. The simplest case is that of infinite trees with no leaves (the root is never considered as a leaf). 

On such a tree $T$, each node has some descendants: for each $h\geq 0$, one can construct a choice function $x_{h}:T\to T$ which associates with each node one of its descendant having height at least $h+1$. 
\begin{pro}\label{pro:jptrj}Let $\U$ be a \AUD irreducible transition matrix on a tree $T$ with no leaves. For all $\lambda\in \mathbb{C}$, the vector $\pi$ defined by $\pi^{(\lambda)}(\root)=1$ and
  \[\pi^{(\lambda)}(u):=\pi^{(\lambda)}(\root) \frac{\det\l((\lambda \Id-{}^u\U)^{(u)}\r)}{\dis\prod_{v \in \cror{  \root,  u}} \U_{v,p(v)}},~~\textrm{ for } u\in T\] is a $\lambda$-left eigenvector of $\U$; hence the spectrum of $\U$ is $\C$.  
\end{pro}
Again, the proposition does not state the uniqueness of the dimension of the $\lambda$-left eigenspaces (and again, as in \Cref{pro:qff}, the dimension of these eigenspaces is $Q(T)$).
    \begin{proof} 
It suffices to prove the statement for $\lambda\geq 1$: indeed, since for all $u\in T$, $\lambda\mapsto \pi^{\lambda}(u)$ is holomorphic, if the finite sum of holomorphic functions
\[\pi^{(\lambda)}(u)-\sum_{v\in \cro{\root,u}\cup c_u(T)}\lambda \pi^{(\lambda)}(v) \U_{v,u}\]
is null for $\lambda \geq 1$, it must be zero on all $\C$ by an analytic continuation argument.

Assume then that $\lambda\geq 1$, choose some integer $h>0$, and consider the irreducible \AUD transition matrix $\U^{h,\lambda}$ on $T$ defined by
	\[\U^{h,\lambda}_{u,v}=\frac{\U_{u,v}+(\lambda-1)\1_{v=x_h(u)}}\lambda\geq 0,\]
    whose $h$-invariant distribution is
	\ben\label{eq:qdgr} \pi^{h,\lambda}(u)=\pi^{h,\lambda}(\root) \frac{\det\l((\Id-{}^u\U^{h,\lambda})^{(u)}\r)}{\dis\prod_{v \in \cror{  \root,  u}} \U_{v,p(v)}^{h,\lambda} }.\een
	Notice that  $\pi^{h,\lambda}(u)$ is constant for all $h$ such that $h> |u|$, since in words, ${}^{u}\U_{a,b}^{h,\lambda}$ for $|a|,|b|$ and $|u|< h$, is defined using a kind of projection onto $\cro{\root, u}$ (and this projection is the same for all $h$ big enough, if $a,b,u$ are fixed).
		Because of this, for all $u$, the sequence $(\pi^{h,\lambda}(u),h\geq 0)$ is eventually constant, and thus
	\[\pi(u):=\lim_h  \pi^{h,\lambda}(u)= \pi(\root) \frac{\det\l((\Id-{}^u\U/\lambda)^{(u)}\r)}{\dis\prod_{v \in \cror{  \root,  u}} \U_{v,p(v)}/\lambda }=\pi(\root) \frac{\det\l((\lambda \Id-{}^u\U)^{(u)}\r)}{\dis\prod_{v \in \cror{  \root,  u}} \U_{v,p(v)}}.\]
	exists, as well as
	\[\lim_h {}^u\U^{h,\lambda}={}^u\U / \lambda.\] 
	Since   ${}^u\U^{h,\lambda}$ and $\pi^{h,\lambda}(u)$ are eventually constant (as a function of $h$), $\pi^{h,\lambda}$ invariant by  $\U^{h,\lambda}$, for a fixed $u$, and $h\to+\infty$ gives $\pi= (\U/\lambda) \pi$, that is $\pi$ is a $\lambda$-left eigenvector of $\U$ (since $\pi^{h,\lambda}(u)=\sum_{v\in \cro{\root,u}\cup c_u(T)} (\U_{v,u}^{h,\lambda}/\lambda) \pi^{h,\lambda}(v)$ involve a finite number of nodes).
\end{proof}

When a tree is reduced to a single infinite path starting at $\root$, any irreducible \AUD transition matrix admits all complex numbers as  simple left eigenvalues (see \cite[Theo. 3.8]{MR4669754}).

\subsection{More general spectral properties}

\label{sec:MG}

\paragraph{On a finite space $F$,} if $\M$ is an irreducible transition matrix:\\
\bls First, by the Perron-Fröbenius theorem, there exists a single invariant measure that can be normalized to be a probability distribution: the total mass is 1.\\ 
The eigenvalues other than 1 have an absolute value strictly less than 1. If $\rho^{\lambda}$ is a $\lambda$-left eigenvector for some $\lambda\neq 1$,
then $S:=\sum_{i\in F}\rho^{\lambda}_i=0$ (since summing the identity , $\sum_i \rho^{\lambda}_i \M_{i,j} = \lambda \rho_j$ over $j$ gives $S=\lambda S$).\par

This last conclusion holds more generally when $T$ is infinite under additional assumptions (e.g.,  by Fubini's theorem, when $\sum_{i\in F}|\rho^{\lambda}_i|<+\infty$).\par
When the state space is infinite, as discussed in the previous subsection,  complex numbers of any absolute value can be eigenvalues.

In Section \ref{eq:qsgg}, we saw an example of \AUD transition matrix $\U$ on a finite tree having an eigenvalue $\lambda\neq 1$ with multiplicity $2$, and with infinite multiplicity on a infinite tree. 
As we will see now, the presence of leaves is an obstacle for certain complex numbers to be left eigenvalues (so that the ``full left spectrum'' of \Cref{pro:jptrj} does not extend to all trees). \par
Consider $\U$ an irreducible \AUD transition matrix on an infinite tree $T$, and assume that $T$ has some leaves, which implies also that some nodes $v$ of $T$ are roots of finite subtrees $T_v$.
If $\rho^{(\lambda)}$ is a $\lambda$-left eigenvector of $\U$, we have $\lambda\rho^{(\lambda)}=\rho^{(\lambda)} \U$, and this implies that
\[\lambda \rho_{T_v}^{(\lambda)}= \rho_{T_v}^{(\lambda)} \U_{T_v,T_v}+ \rho_{\cro{\root,p(v)}}^{(\lambda)}\U_{\cro{\root,p(v)},T_v}\]
and then
\begin{align}\label{sys:rev}
	\rho_{T_v}^{(\lambda)}(\lambda \Id-\U_{T_v,T_v})= \rho_{\cro{\root,p(v)}}^{(\lambda)}\U_{\cro{\root,p(v)},T_v}.
\end{align}
The matrix $\lambda \Id-\U_{T_v,T_v}$ is invertible iff $\lambda$ is not an eigenvalue of $\U_{T_v,T_v}$. As in the proof of \eref{pro:qff}, this implies that any pre-solution of $\rho^{(\lambda)}$ constructed on the branch below $v$ can be extended to $T_v$, by applying the inverse of the matrix $\lambda \Id-\U_{T_v,T_v}$ to the right in equation \ref{sys:rev}. However, when $\lambda \Id-\U_{T_v,T_v}$ is not invertible, a pre-solution on $T\setminus T_v$ can be extended only if $\rho_{\cro{\root,p(v)}}\U_{\cro{\root,p(v)},T_v}$ belongs to the vector space generated by the rows of $(\lambda \Id-\U_{T_v,T_v})$ (because for any row vector $R$, $R(\lambda \Id-\U_{T_v,T_v})$ lives in the vector space generated by the rows of $(\lambda \Id-\U_{T_v,T_v})$).

Hence, there are two cases: either the pre-solution cannot be extended to $T_v$, either the set of extensions form an affine subspace with dimension $\dim\l(\Ker \l(\lambda \Id-\U_{T_v,T_v}\r)\r)$.
As in the proof of  Proposition \eref{pro:qff}, the extension of a pre-solution along an infinite path, or towards a node from which hang several infinite subtrees, is always possible. 
Since an infinite tree $T$ (locally finite) has countably many nodes, we have:
\begin{pro} Let $\U$ be an irreducible \AUD transition matrix on an infinite tree $T$, then the set of complex numbers $\lambda$ such that $\lambda$ is not a left eigenvalue of $\U$, is countable.
\end{pro}

\color{black}

\subsection{Random walks on trees}
\label{sec:RWT}

  \paragraph{Related works}
One can find, in the literature many papers about random walks on either complete trees (in which all nodes have degree $d$), or on super-critical Galton-Watson trees conditioned on non-extinction. Among others, Lyons \& al. \cite{MR1410689},Peres \&  Zeitouni \cite{MR2365486}, 
(see also Aid\'ekon \cite{MR2438700} for the random environment case, in which the edges of a tree, without leaves, are given different random probabilities; see also Ben Arous \& Hammond \cite{MR2969494}, Hammond \cite{MR3098688}).
Some questions on these models were published in 1994 (by Lyons \& al. \cite{MR1601753}) and many of them are still open.

Duquesne \cite{MR2184096} studies the set of visited nodes of a family of transient random walks on a complete $d$-ary tree: the set of visited nodes forms a subtree $T_p$ of the initial $d$-ary tree. The family of random walk considered is indexed by a parameter $p$, which is taken close to a critical point (for which the rate of convergence of the chain to $+\infty$ goes to zero). Duquesne shows that $T_p$ suitably rescaled, converges in distribution to a limiting continuum random tree (which is infinite, and two-sided, in the sense that two contour processes are needed to describe it).

In the case of random walks on trees, the $h$-invariant measure (see \Cref{theo:1.3}) admits a simpler form:
for any node $u=u_1\cdots u_h$  of a tree $T$ we denote by $u[j]=u_1\cdots u_j$ the ancestor of $u$ at level $j$, with $u[0]=\root$. A consequence of \Cref{pro:dadgty}$(ii)$,
\begin{pro}
If $\M$ is the transition of an irreducible random walk on $T$, then the $h$-invariant measure of $\M$ is $\bar\pi=(\bar\pi_u,u\in T)$ defined by  $\bar{\pi}(\root)=1$, and for $u\in T$,
\ben\label{eq:pibar}
\bar{\pi}_u=\bar{\pi}(\root)\prod_{j=1}^h\frac{ M_{u[j-1],u[j]}}{ M_{u[j],u[j-1]}},~~\textrm{ for } u\in T.
\een
Therefore, $\M$ is positive recurrent, if and only if $\sum_{u \in T} \bar{\pi}_u<+\infty$. 
\end{pro}
Again, \Cref{pro:qff} still applies here: if a tree have several ends, here also  the 1-left eigenspace has dimension greater than 1. The cone generated by invariant measures may also have dimension greater than 1 (see e.g. \Cref{sec:qfgrht}).
\begin{proof} The fact that $\bar{\pi}$ is invariant by $\M$ can be proved by noticing that $\bar{\pi}$ is reversible with respect to $\M$, which is immediate. One can also use the formula \eref{eq:fgeht}, and observe that the sum in the right hand side simplifies since only transitions to neighbors are possible here.~\\
  Since RW are special case of \AUD transition matrices, the measure $\pi$ given in Theorem \ref{theo:1.3}, is also invariant by $\M$. As a matter of fact, $\pi$ and $\bar{\pi}$ coincides even if it is not visible. The denominator of the rhs of $\bar\pi$ in \eref{eq:pibar} equals $\prod_{w \in \cror{ \root, u}} M_{w,p(w)}$, so that it suffices to observe that the total weight of $W\l(\cro{\root,p(u)},T_u\r)$ coincides with the numerator: this is a consequence of the fact that the only spanning tree of the branch $(\cro{\root,p(u)},T_u)$ using the transition with positive weight is the path $(\root,u[1],\cdots,u[h-1],u)$. But the transition from $u[h-1]=p(u)$ to $u$ is weighted $M_{p(u),T_u}$.
\end{proof}

\subsection{Random walk in the  critical Galton-Watson tree conditioned on non-extinction}

Consider a Galton-Watson tree ${\bf T}$ with a critical offspring distribution $p=(p_k,k\geq 0)$ meaning that $\sum_{k\geq 0} kp_k=1$; let us assume further that  $p_0+p_1<1$. Denote by $Z_n$ the population in the $n$th generation of ${\bf T}$.
The law ${\cal L}({\bf T}~|~Z_n>0)$ converges when $n\to+\infty$ for the local topology, towards the law of an infinite tree ${\bf T}^{\infty}$ with a single end (see Lyons \& al. \cite{LPP}). The tree ${\bf T}^\infty$ is often referred to the critical Galton-Watson tree conditioned on non-extinction, in the literature. To define  ${\bf T}^\infty$,  we first describe its end $\p:=(u_0=\root,u_1,\cdots,)$ which is almost surely unique. To produce the end, sample a sequence $(c_j,j\geq 0)$ of iid random variables taken under the biased offspring distribution $\widehat{p}:=(kp_k,k\geq 0)$. Then, $c_j$ children are successively ascribed to $u_j$ (with $u_0=\root$), and a uniform node among them is chosen as node $u_{j+1}$. This produces an infinite path (the end $\p$), each node of which has an additional random progeny: the set of these additional nodes forms the neighborhood of $\p$.
Then, to produce ${\bf T}^\infty$, we graft at each neighbor of the end $\p$, iid (critical) Galton--Watson trees with offspring distribution $(p_k)$.

Now, consider a \underbar{random walk} on ${\bf T}^\infty$, with a general irreducible transition matrix $\M$. 
Here we see that $Q(\bT^\infty)=1$ (as defined in \eref{eq:QT}), which gives the uniqueness of the $1$ left eigenvector (up to a constant). In other words, the $h$-invariant measure $\pi$ generates the 1-left eigenspace.\par
    To decide the positive recurrence is more demanding, since it is equivalent to the finiteness of  
\[{\sf Total}:=\sum_{u\in T} \pi(u).\]
\paragraph{A class of solvable model.} We will add more hypothesis on $\M$ to exhibit a class of cases where the question can be treated with a simple argument.

We will make a quite common assumption: we will assume that the transition matrix $\M$ on  $T:={\bf T}^{\infty}$ is homogeneous, in the sense that $M_{u,v}$ depends only on the degree $|c_T(u)|$ of $u$ and the ``direction'', towards the root or  not. Take two sequences $(G(k),k\geq 0)$ and $(F(k),k\geq 0)$ of positive real numbers satisfying moreover
\ben\label{eq:bore} 1-  F(k)-kG(k)\geq 0,~~~~\textrm{ for all } k\geq 0.\een
Now, we require that:\\
-- for each node $u\in T$, $u\neq \root$, $\M_{u,p(u)} = F(|c_T(u)|)$,\\
-- if $v$ is a child of $u$, $\M_{u,v}=G(|c_T(u)|)$,\\
-- if $u\neq \root$, $\M_{u,u}= 1-  F(|c_T(u)|)-|c_T(u)|G(|c_T(u)|)$
which is non-negative by \eref{eq:bore}.~\\
For the root, assume that $\M_{\root,v}=G(|c_T(\root)|)$ for all children of $\root$, and $\M_{\root,\root}=1-|c_T(\root)|G(|c_T(\root)|)$.
\begin{pro}\label{pro:sqfd} For $X$ a random variable $p$ distributed, set 
\be f &:=& \E(X G(X)/F(X)),\\
    m &:=& \frac{1}{1-f} \E(1/F(X)),\\
    L &:=& \E\l(X\log(G(X)/F(X))\r).
\ee
\bir
\itr If $f \in (0,1)$, $ \E(1/F(X))<+\infty$, $L<0$, and there exists $x\in (0,1)$ such that, for a $\widehat{p}$-distributed random variable $Y$,
\[\sum_n \P\l( \log\l[\frac{ 1+   (Y-1)G(Y)m}{F(Y)}\r]\geq nx L\r)<+\infty,\]
then $M$ is almost surely positive recurrent. 
\itr If $L>0$ then $M$ is almost surely (also denoted a.s.) not positive recurrent.
\eir 
\end{pro}
\begin{proof}
  $(i)$ We will prove that when the hypothesis are satisfied, $\E({\sf Total})<+\infty$, so that ${\sf Total}<+\infty$ a.s.
  
  We still denote by $\s(a,u)$ the successor of $a$ in the direction of $u$, by \eref{eq:pibar} (setting $\pi_\root=1$),
\ben\label{eq:ehtyrjui}
\pi(u) &=& \frac{\M_{\root,s(\root,u)}}{\M_{u,p(u)}} \prod_{v\in\croc{\root,u}} \frac{\M_{v,\s(v,u)}}{M_{v,p(v)}}
      = \frac{G(|c_T(\root)|)}{F(|c_T(u)|)} \prod_{v\in\croc{\root,u}} \frac{G(|c_T(v)|)}{F(|c_T(v)|)}.
\een
Along the infinite branch, the $|c_T(v)|$ are iid and distributed according to $\widehat{p}$. If one takes the node $u_h$ in the end at level $h$, the value of $\pi(u_h)$ is (up to the first factor), a product of iid random variables, quite easy to study.  However, the summation of $\pi(u)$ must be done on  $\bT^\infty$, not only on $\p$: even if all the subtrees rooted at the neighbors of $\p$ have the same distribution, and if their sizes are finite a.s., each of them has infinite size in expectation, which means that we must take them into account.

A node $u\in \p$ such that $|c_T(u)|=k$, has $k-1$ children out of $\p$, says  $v_1(u),\cdots,v_{k-1}(u)$. The subtree $\bT^\infty_{v_i(u)}$ of $\bT^\infty$ rooted at $v_i(u)$ is a critical $p$-Galton-Watson tree independent from the rest of $\bT^\infty$. 

Let us treat all the invariant measures of all these trees separately, as if they were disconnected from $\bT^\infty$. Denote by $\rho^{v_i(u)}$ the invariant measure of $\M$ over $\bT^\infty_{v_i(u)}$ (the subtree of  $\bT^\infty$ rooted at $v_i(u)$).\par
Take a critical Galton-Watson tree $t_{GW}$, equipped with the  transition matrix $\M$, and let $\rho$ be its invariant distribution normalized by $\rho_{\root}=1$. Consider the total mass $S_{GW}=\sum_{u\in t_{GW}} \rho_u$ of the invariant measure on $t_{GW}$. Denote by $S_{GW}[{d}]$ a random variable distributed as $S_{GW}$, conditioned by $|c_{t_{GW}}(\root)|=d$. 

First, the $h$-invariant measure $\pi$ on $\bT^\infty$ restricted to the subtree $\bT^\infty_{v_i(u)}$ satisfies: for all $v_i(u)$,
\ben \l(\pi(w), w \in \bT^\infty_{v_i(u)}\r)&=& \pi(u)\frac{\M_{u,v_i(u)}}{\M_{v_i(u),u}}\l(\rho^{v_i(u)}_w,  w \in \bT^\infty_{v_i(u)}\r)
=\pi(u)\frac{ G(|c_T(u)|)}{F(|c_T(v_i(u))|)}\l(\rho^{v_i(u)}_w,  w \in \bT^\infty_{v_i(u)}\r),\een
which implies that
\[\sum_{w \in T_{v_i(u)}} \pi(w) =  \pi(u)G(|c_T(u)|)
  \frac{S^{v_i(u)}[|c_T(v_i(u))|]}{F(|c_T(v_i(u))|)},\]
where the $S^{v_i(u)}[|c_T(v_i(u))|]$ are independent copies of $S_{GW}[|c_T(v_i(u))|]$;  now compute the total mass of the invariant measure ${\sf Total}$ by assembling the subtrees hanging from the same nodes of $T$, 
\ben{\sf Total}&=&\label{eq:2-total}\sum_{u\in \p}\pi(u)\l[1+ \sum_{i=1}^{|c_T(u)|-1}   G(|c_T(u)|)\frac{ S^{v_i(u)}[|c_T(v_i(u))|]}{F(|c_T(v_i(u))|)} \r].
\een
Hence 
\ben \label{eq:total}
\E\l[{\sf Total}~|~ (c_T(u),u\in \p)\r] &=&\sum_{u\in \p}\pi(u)\big[1+ (|c_T(u)|-1) G(|c_T(u)|) m \big]\een
where
\[m = \E\l[\frac{ S^{v_i(u)}[|c_T(v_i(u))|]}{F(|c_T(v_i(u))|)} \r],\]
and this last quantity does not depend on $u$ since $|c_T(v_i(u))|$ is the number of children of the descendants of $u$ (they are $p$-distributed). 

We have for all $d\geq 0$, by decomposing a Galton-Watson tree with degree $d$ at its root (and then divide everything by $F(d)$, by convenience),
\ben\label{eq:sf} \frac{S_{GW}[d]}{F(d)} \eqd \frac{1}{F(d)}+\sum_{k=1}^d \frac{G(d)}{F(d)} \frac{S_k[X_k]}{F(X_k)},\een
where the $(X_i)$ are iid and follow the offspring distribution $(p_k)$, and where the $S_{GW}^{(k)}[X_k]$, conditioned by $(X_k, k=1,\cdots,d)=(x_k, k=1,\cdots,d)$, are independent copies of $S_{GW}[x_k]$, respectively.

Taking the expectation at both side in \eref{eq:sf}, gives
\ben \label{eq:m} m = \sum_{k\geq 0}\frac{p_k}{F(k)}+ m \sum_{k\geq 1} k p_k\frac{G(k)}{F(k)},  \een
so that, since 
\[\sum_{k\geq 1} k p_k \frac{G(k)}{F(k)}=\E\l(\frac{XG(X)}{F(X)}\r)=f\]
and  by hypothesis $f\in (0,1)$ and $\E(1/F(X))<+\infty$, so that \eref{eq:m}, leads to
\ben m = \frac{1}{1-f}\sum_{k\geq 0}\frac{p_k}{F(k)}= \frac{1}{1-f}\E(1/F(X))<+\infty.\een
Now, we can plug again the value of $\pi(u)$ on the end $\p$:
\ben\label{eq:rshtry} \E\l[{\sf Total}~\Big|~ (|c_T(u)|,u\in \p)\r]=  G(|c_T(\root)|)\sum_{u\in \p}\l( \prod_{v\in\croc{\root,u}} \frac{G(|c_T(v)|)}{F(|c_T(v)|)}\r)\frac{ 1+   (|c_T(u)|-1)G(|c_T(u)|)m}{F(|c_T(u)|)}\een
By using the fact that the degree on $\p$ are iid and follow the biased distribution $\widehat{p}$ and the strong law of large number, denoting by $u_n$ the node at distance $n$ of $\root$ on $\p$,
the logarithm of the contribution of term $u=u_n$ in the sum, in the right hand side of \eref{eq:rshtry}
\[\log\l[\l(\prod_{v\in\croc{\root,u_n}} \frac{G(|c_T(v)|)}{F(|c_T(v)|)}\r)\r]+\log\l[\frac{ 1+   (|c_T(u_n)|-1)G(|c_T(u_n)|)m}{F(|c_T(u_n)|)}\r].\]
The first term behaves as $n \E\l(\log\l(\frac{G(Y)}{F(Y)}\r)\r)$ by the law of large number 
where $Y$ is $\widehat{p}$-distributed, and we have 
\[
\E\l(\log\l(\frac{G(Y)}{F(Y)}\r)\r)=\E\l(X\log\l(\frac{G(X)}{F(X)}\r)\r)=L.
\]
By hypothesis  $L<0$, so that  $n \E\l(\log\l(\frac{G(Y)}{F(Y)}\r)\r)$ goes linearly to $-\infty$.

In order to conclude we need that $\log\l[\frac{ 1+   (|c_T(u_n)|-1)G(|c_T(u_n)|)m}{F(|c_T(u_n)|)}\r]$ is not too often ``larger'' than $nL$, and the existence of $x\in (0,1)$ such $\sum_n\P\l( \log\l[\frac{ 1+   (Y-1)G(Y)m}{F(Y)}\r]\geq nx L\r)<+\infty$ ensures that this will occur a.s. only a finite number of times by Borel--Cantelli Lemma. \\ 
$(ii)$ If $L>0$, for $u=u_n$ be the $n+1$-th node on $\p$, using \eref{eq:ehtyrjui}, we see that $\log(\pi(u_n))$ goes to $+\infty$ a.s.

\end{proof}

\section{\Cld{} Markov chains on trees}
\label{sec:CLD}

The main message we want to convey here is that \ALD Markov chains are more complex than \AUD Markov chains. There are several reasons for this, some of them have been explained in \Cref{sec:W}. Even if the tree is a single line (isomorphic to $\N$), \ALD transition matrices $D$ may have a cone of invariant positive measures with any dimension ranging from 0 to $+\infty$ (included) (see \cite[Theorem 3.1]{MR4669754}; still on $\N$, when $\U$ is \AUD, the picture is radically different, since we have a formula to compute the unique invariant measure.

On infinite general trees, the same picture can be drawn with some variations (the picture is already quite complex in the case of \AUD).

Nevertheless, since lines are infinite trees, the example in \cite[Theorem 3.1]{MR4669754} of \ALD Markov chain having no invariant measure exemplified the same property in infinite trees (and the same holds for the examples with any finite or infinite dimensions even on a single line, that is, on a tree with a single end).

However, some analogous of \cite[Theorem 3.1]{MR4669754} can be stated here in the general case of \ALD transition matrices: the main driving idea is that, the time reversal of any \AUD Markov chains with respect to some invariant measure (such a measure exists) is a \ALD Markov chain (having the same invariant measure). The other way round does not hold in general since the existence of an invariant measure of \ALD transition matrices is not ensured; when there are several invariant measures, there are several ways to define the time reversal processes (each of them being a \AUD transition matrix).
\begin{theo} Let $\;\U$ be an irreducible \AUD transition matrix on a finite or infinite tree $T$, and let $\pi:=(\pi_v,v\in T)$ be an invariant measure of $\;\U$. Set, for all $u,v \in T$,
  \ben \DD_{v,u}=\pi_u\U_{u,v}/\pi_v, ~~~u,v\in T.\een
  \bir
  \itr $\DD$ is an \ALD irreducible transition matrix, which possesses also $\pi$ as invariant measure
  \itr $\DD$ is positive recurrent iff $\U$ is,
  \itr $\DD$ is recurrent iff $\U$ is,
  \itr If $\U$ has a right eigenvector $R=(R_u,u\in T)$ with positive coordinate (and associated with the eigenvalue 1) then $\bma \pi_u R_u,u\in T\ema$ is an invariant measure for $\DD$ (so that this property extends to cones). 
  \eir
\end{theo}
[The proof are similar to those of  \cite[Theorem 3.1]{MR4669754}].
\begin{proof} $(i)$ this is standard, $(ii)$ the positive recurrence is equivalent to the fact that there exists an invariant measure with is summable $\sum \pi_i<+\infty$ (in which case uniqueness follows), so that it is clear, $(iii)$ the recurrence can be expressed in terms of the total weights of the paths starting at $\root$ and ending at $\root$: it is equivalent, for $\U$ to the fact that
  \ben\label{eq:fhetr} \sum_{k} \sum_{(p_0,\cdots,p_k)} \prod_{i=0}^{k-1}\U_{p_i,p_{i+1}}=1\een where the sum is taken on all $k\geq 1$, and all paths $(\root=p_0,p_1,\cdots,p_k=\root)$ with $k$ steps such that $p_i\neq \root$ for $i\in\{1,\cdots,k-1\}$. Since the starting and ending points $p_0$ and $p_k$ are the same, we have $\sum_{k} \sum_{(p_0,\cdots,p_k)} \prod_{i=0}^{k-1}\U_{p_i,p_{i+1}}=\sum_{k} \sum_{(p_0,\cdots,p_k)} \prod_{i=0}^{k-1}\DD_{p_i,p_{i+1}}$ and the conclusion follows. $(iv)$ Since $\sum_b \U_{a,b}R_b=R_a$, replacing $\U_{a,b}$ by $\DD_{b,a}\pi_b/\pi_a$ in this equation gives $\sum_b \pi_b \DD_{b,a}R_b=\pi_aR_a$ as desired.
\end{proof}

\section{Combinatorial and algebraic aspects of \AUD Markov chains}
\label{sec:CAA}
    
\subsection{Multicontinuous fractions and Markov chains on trees}

\subsubsection{Continuous fractions, Mötzkin paths, and birth and death processes}
  
Going back to Euler \cite{Euler} (see Sokal \cite{MR4602842} for a recent historical account), the continuous fractions are used in mathematics in various contexts, notably, to provide an alternative representation of real or complex numbers. By appropriate use of some associated toolboxes, it gives access to some unavailable information that is useful in many applications. For example, the proof of the irrationality of $\zeta(3)$ by Apéry \cite{MR3363457} relies on the representation $\zeta(3)$ as a continuous fraction \[\zeta(3)=\cfrac{6}{w_0-\cfrac{1^6}{w_1-\cfrac{2^6}{w_2-\cfrac{3^6}{\cdots}}}},~~\textrm{with}~~~w_k=(2k+1)(17 k(k+1)+5),\]
although Apéry did not use this formula, he worked with the standard convergents of this continuous fraction, and used the standard speed of convergence argument, designed for continuous fractions, which provides a way to prove irrationality under certain conditions.\par
  Continuous fraction representations also play an important role in the analysis of birth and death processes. The main technical point, discovered by  Flajolet \cite{Flajolet} (and pursued in a series of papers including \cite{FG}), is that some formal series associated with some families of weighted paths (as generating functions), satisfy some algebraic decompositions, which leads to the fact that they can be expanded as formal continuous fractions.

 Let us be a bit more explicit here. A Mötzkin path of length $k$ is a path $p=(p_0,\cdots,p_k)$ with $k$ steps, with non negative entries $p_i\geq 0$, and increments $p_{i+1}-p_{i}\in\{-1,+1,0\}$; the length of a path $\ell(p)$ is its number of steps. The set of all  Mötzkin paths (regardless of their lengths)  but starting at $p_0=0$ is denoted ${\cal M}_{0,*}$.
 
 Given some weights $(w_{a,b}, a,b \in \mathbb{N}, |a-b|\leq 1)$ associated with the steps, the weight of a path  $W(p)$ is defined to be simply $W(p)=\prod_{j=1}^{\ell(p)} w_{p_{j-1},p_j}$, and the total weight of a set $S$ of \Motzkin paths is defined to be
 \[{\cal W}(S)=\sum_{p\in S} W(p).\]
 When $S={\cal M}_{0,*}$  (or any other subset of it), the series ${\cal W}(S)$ is well defined under some conditions on the weights: Of course, it may be natural to take complex or real numbers (under convergence conditions), but  formal variables such as $w_{a,b}=x \bar{w}_{a,b}$ (where $x$ is a formal variable used to count the number of steps and $\bar{w}_{a,b}\in \C$) are often preferred since the formal power series of a set of paths containing a finite number of paths of each length is well defined; this is the case, for example, for the set of all Mötzkin paths with a given starting point.
 
Flajolet \cite{Flajolet} also considers non-commutative formal variables $w_{a,b}$ (we will come back to this shortly in order to establish a formula of the greatest degree of generality). 

\paragraph{Continuous fractions.}

Let ${\cal M}_{h,h}$ be the set of Mötzkin paths starting and ending at level $h$, and never going below this level
\ben
{\cal M}_{h,h}= \cup_{k\geq 0} \{ p : (p_0,\cdots,p_k),~ |p_i-p_{i-1}|\leq 1, p_0=p_k=\min (p)=h\}.
\een
Each path $p$ in the set of Mötzkin paths ${\cal M}_{0,h}$ (starting at 0 and ending at level $h\geq 0$, can be seen as the concatenation of a path of ${\cal M}_{0,0}$, then a step $(0,1)$, then a path from ${\cal M}_{1,1}$, then a step $(1,2)$ (when $h\geq 2$),...
Finally, in terms of generating functions it gives
\[\mathcal{W}({\cal M}_{0,h})=\l[\prod_{j=0}^{h-1} \mathcal{W}({\cal M}_{j,j})w_{j,j+1}\r]\mathcal{W}({\cal M}_{h,h}).\] 
These weighted paths are of particular interest for birth and death Markov chains, in which case an history of the population is a weighted Mötzkin path, and the probability is given by the product of the transitions (that is $(w_{a,b})=(M_{a,b}))$; this can be enriched with formal variables, in order to produce generating functions, counting for example, the number of steps, or/and the passage number at some positions.
For example, since the total weight $\mathcal{W}({\cal M}_{0,0})=\sum_{p\in {\cal M}_{0,0}} W(p)$ is related to the number $N$ of passages at zero of the $w$-Markov chain (when $w$ is a transition matrix $w=M$)
\ben\label{eq:rec1} \mathcal{W}({\cal M}_{0,0})=\sum_{i=0}^{+\infty} \E_0(\1_{X_i}=0)=\E(N).\een
When the weight is $w_{a,b}=M_{a,b}x$ instead, we get
\ben\label{eq:rec2} \mathcal{W}({\cal M}_{0,0})=G_0(x)=\sum_{p: p\in{\cal M}_{0,0}} W(p)=\E(x^{N-1})\een
where $G_0(x)$ is the generating Green function, which measures the total number of passage to $0$, starting from 0. Formula \eref{eq:rec1} is crucial for establishing recurrence, and \eref{eq:rec2} is crucial for positive recurrence.

The fundamental decomposition of Mötzkin path leads to continuous fractions as follows:\\
-- Any element $p$ in ${\cal M}_{h,h}$ can be either:\\
$\bullet$ reduced to the zero length path $p=(h)$,\\
$\bullet$ or, starts by an horizontal step followed by an element of ${\cal M}_{h,h}$,\\
$\bullet$ or, starts by a up step, followed by an element of ${\cal M}_{h+1,h+1}$, and then, a down step, followed by an element of ${\cal M}_{h,h}$.\\
Setting $W_h:={\cal M}_{h,h}$ we have,
\ben\label{eq:dqg} W_h=1+ w_{h,h}W_h+w_{h,h+1}W_{h+1}w_{h+1,h}W_h.\een
One then has, when everything is well-defined,
\ben W_h=\frac{1}{1-w_{h,h}-w_{h,h+1}w_{h+1,h}W_{h+1}},~~\textrm{ for all }h\geq 0,\een
and finally, this gives, by iteration of this formula
\ben\label{eq:qgqd}
  W_0= \cfrac{1}{1-w_{0,0}-\cfrac{w_{0,1}w_{1,0}}{1-w_{1,1}-\cfrac{w_{1,2}w_{2,1}}{1-w_{2,2}-\cfrac{w_{2,3}w_{3,2}}{1-w_{3,3}\cdots}}}}
\een
Flajolet \cite{Flajolet} noticed that the formula for Mötzkin paths staying below level $h$ is just given by the $h$-th convergent of \eref{eq:qgqd} (where only appear the weights of the steps which do not touch level $h+1$, and of course, this is consistent with the fact that taking zero as the weight of a step, just forbids to a walker to take that step) \footnote{this formula allows to compute the probability that a birth and depth process comes back at zero before reaching a given level $h$.};
he also explains that, for example, replacing $w_{a,b}$ by $w_{a,b}y^b$ for a formal variable $y$, amounts to multiplying the global weight of a path $p$ by $y^{\sum p_i}$: hence, the generating function of the path, is enriched, and now, there is an additional variable $y$ at the exponent the ``discrete area'' $\sum p_i$ of $p$ and the corresponding generating function possesses then a continuous fraction representation. 
Furthermore, the treatment of special cases provides nice formulas, and identities. For example, the generating function of Dyck paths (a particular case of Mötzkin paths, where the increments are either $+1$ or $-1$, with weight $w_{a,b}=x\, 1_{|a-b|=1}$), is known to solve $G(x)=1+x^2G(x)^2$, because a Dyck path, is either of length 0, or can be written $u D_1dD_2$ where $u$ and $d$ are a step $(0,1)$ and $(1,0)$ and $D_1$, $D_2$, the natural decomposition with a pair of Dyck paths. The solution is $G(x)=\frac {1}{2\,{x}^{2}} \left( 1-\sqrt {1-4\,{x}^{2}} \right) $, and by \eref{eq:qgqd}, with $w_{i,i}=0$, $w_{i,i-1}=w_{i,i+1}=x$, we get
\ben
\frac {1}{2\,{x}^{2}} \left( 1-\sqrt {-4\,{x}^{2}+1} \right) =\cfrac{1}{1-\cfrac{x^2}{1-\cfrac{x^2}{\cdots}}}.
\een
In other words, each family of Mötzkin paths for which another formula is available provides such an identity.

These analytic formulas are exploited in \cite{Flajolet} who provides many identities of this kind (see also Conrad \& Flajolet \cite{CF}, Flajolet \& Guillemin \cite{FG} with a special focus on the connection with birth and death processes (see in particular Table 3).
Flajolet \cite{Flajolet} states his first result using non-commutative variables, and formal expansion of continuous fractions using non commutative variables. In short, the combinatorial properties of the Mötzkin path, encode the statistical properties of birth and death processes (in particular Green functions), which are, in turn, represented by continuous fractions, allowing many connections with other branches of mathematics (the abstract of the paper by Flajolet \cite{Flajolet} gives an idea of the depth of these connections).

\subsubsection{Random walks on trees and multicontinuous fractions}

A large part of the previous discussion on Mötzkin paths can be generalized  for paths on trees (some of the considerations discussed below can be found in Varvak \cite{AV}, up to some change of notation and perspective; other occurrences  of multicontinuous fractions in relation with combinatorial objects can be found in Albenque \& Bouttier \cite{AB}).

Let us start this discussion with \textbf{a random walk on an infinite $d$-ary tree $T$} (where each node has exactly $d$-children) for some $d\geq 2$.
 
Let us define the set of generalized Mötzkin paths ${\cal M}_{u,u}$, as the set of paths of any length $p=(p_0,\cdots,p_k)$  in the subtree $T_u$ (rooted at $u$), that start and end at $u$, and for which for all $i$, $p_i$ and $p_{i+1}$ are neighbors in the tree, or possibly $p_{i}=p_{i+1}$ (these are random walk trajectories). Assuming that each step $(p_i,p_{i+1})$ has weight $w_{p_i,p_{i+1}}$ for a certain weight function, the corresponding generating function $W_u$ (the sum  of the weights $W(p)=\prod_{j=1}^{\ell(p)} w_{p_{j-1},p_j}$ taken for all paths $p$ in ${\cal M}_{u,u}$), satisfies a recursion close to \eref{eq:dqg}:
\ben
W_u=1+w_{u,u}W_u+\sum_{v \in c_T(u)} w_{u,v}w_{v,u} W_vW_u
\een
and again, these formulas may be valid in $\C$ or in a formal sense (or both) depending on the convergence properties of the generating functions $W_v$ and $W_u$.
This gives
\ben\label{eq:tree_CT}
W_u=\frac{1}{1-w_{u,u}-\sum_{v \in c_T(u)} w_{u,v}w_{v,u} W_v}
\een
and again, one can expand $W_v$ in the same way, according to its children, and get
\ben\label{eq:tree_CT2}
  W_u=\cfrac{1}{1-w_{u,u}-\sum_{v \in c_T(u)} w_{u,v}w_{v,u} \cfrac{1}{1-w_{v,v}-\sum_{w \in c_T(v)} w_{v,w}w_{w,v}W_w}};
\een
this expansion is done by applying recursively \eref{eq:tree_CT}, and replacing the term $W_{n}$ relative to nodes $n$ at a certain depth, by a fractions depending on the $W_m$, where the $m$ are the children of $n$. 

Again, the $h$-th convergent of this fraction is the generating function of paths staying below level $h$.

\paragraph{Green function}  If $w$ is a Markov chain transition matrix on $T$, the associated Green's generating function $G_u(x)=\sum_{t\geq 0} \E_u( 1_{X_t=u} x^t)$ can be expressed as in \eref{eq:rec2}. 
For a $w$-Markov chain let $RT(u)=\inf\{t>0: X_t=u\}$  the first return time to $u$.
Set $H_u(x)= \E_u(x^{RT(u)})$. We have 
\[G_\root(x)= 1+\sum_{k=1}^\infty H_\root(x)^k=\frac{1}{1-H_\root(x)} . \]
In the case of the random walk, by setting for all $u\in T$,
\ben\label{eq:tree_CTp}
G_u(x)=\frac{1}{1-x\M_{u,u}-x^2\sum_{v \in c_T(u)} \M_{u,v}\M_{v,u} G_v(x)}
\een
which allows to identify $H(x)=x\M_{\root,\root}+x^2\sum_{v \in c_T(\root)} \M_{\root,v}\M_{v,\root} G_v(x)$.
The Green function $G_{\root}$ can be used to determine the recurrence, or positive recurrence of the $M$-Markov chain.

\paragraph{Examples}

The following example is given under a general form of weighted paths; we will come back to Markov chains considerations eventually.

Consider an infinite complete binary tree $T$, in which each node has 2 children.
Each node $u=u_1\cdots u_h$ at depth $h$ is a word in $\{0,1\}^h$. Denote by $|u|_0$ the number of zeros in $u$. Consider a random walk on $T$ and weight each oriented edge $(u,v)\in T$ by $M_{u,v}$ (an element of $\C$ or a formal variable).

We assume that the local transitions from all $u\neq \root$, $M_{u,.}$ depend only on $|u|_0$, in the sense that
\ben\label{eq:rlstuff}(M_{u,u0},M_{u,u1},M_{u,u},M_{u,p(u)})=(r_{|u|_0},\ell_{|u|_0},s_{|u|_0},\mathsf{p}_{|u|_0})\een where the variable name are chosen to refer to the moving direction from $u$ which are ``right'', ``left'',''stay'', ``parent''.
Since the root has no parent, we need to relax this assumption at the root, and take new parameters $M_{\root,\root},M_{\root,0},M_{\root,1}$ for the root. 
The set of parameters for this model is $\l[(M_{\root,\root},M_{\root,0},M_{\root,1}), \big((s_k,\ell_k,r_k,\mathsf{p}_k),k\geq 0\big)\r]$. 

The generating function $G_{u}$ of nearest neighbor paths starting at $u$, coming back to $u$, and remaining in $T_u$ (counted according to the number of steps) depends only on $|u|_0$, except for $u=\root$ (since the transition in the subtrees $T_u$ and $T_{u'}$ such that $|u|_0=|u'|_0$ can be identified, by translation).

Denote by $g_{k}$ the common value of the $G_{u}$ with same $|u|_0=k$, except for $u=\root$ which is treated separately,  again using a simple decomposition according to the first step, 
\ben\label{eq:groot} G_{\root}&=&
1+ \l[M_{\root,\root}+ M_{\root,1}{\sf p}_{0}g_{0} +M_{\root,0}{\sf p}_{1} g_{1}\r]G_\root\\
&=& \frac{1}{1- \l[M_{\root,\root}+ M_{\root,1}{\sf p}_{0}g_{0} +M_{\root,0}{\sf p}_{1} g_{1}\r]}.\een
We have, for $k\geq 0$, by counting the evolution of the number of zeros in a node representation when one moves along an edge,
\ben\label{eq:gk}
g_{k}&=&1+ \l[s_k+ \ell_k{\sf p}_kg_{k} +r_k {\sf p}_{k+1} g_{k+1}\r]g_{k},\\
\label{eq:gk2}&=&\frac{1}{1-\l[s_k+ \ell_k{\sf p}_kg_{k} +r_k{\sf p}_{k+1} g_{k+1}\r]},\een
which is a multicontinuous formula for all the $g_k$ and for $G_\root$ (using \eref{eq:groot} too).
But \eref{eq:gk} can be solved in $g_k$ in terms of $g_{k+1}$,
\ben\label{eq:gk3}
g_k:=\frac { 1- {\sf p}_{k+1}\,g_{k+1}\,r_k\,-s_k-\sqrt {{{\sf p}_{k+1}}^{2}{g_{k+1}}^{2}{r_k}^{2}+2\,{\sf p}_{k+1}\,g_{k+1}\,r_k(s_k-1)-4\,{\sf p}_k
\,\ell_k+({s_k}-1)^2} }{2\,{\sf p}_k\,\ell_k}
\een
from it one sees a very different picture of  branching infinite nested radicals of all the $g_i$ (and of $G_{\root}$).

Infinite nested radicals arise in number theory and famous examples are some identities as
\[3=\sqrt{1+2\sqrt{1+3\sqrt{1+4\sqrt{...}}}}\]
due to Ramanujan (to be more precise, $3=\lim_n \sqrt{1+2\sqrt{1+3\sqrt{1+4\sqrt{...+n\sqrt{1}}}}}$ (and this appears as a question in the Journal of the Indian Mathematical Society in 1911, by  Ramanujan).  Identities between continuous fractions and infinite nested radical already exists:
\ben\label{eq:sgef} \phi=1+\cfrac{1}{1+\cfrac{1}{1+\cdots}}={\sqrt{1+\sqrt{1+\sqrt{1+\cdots}}}}\een
where $\phi=(1+\sqrt{5})/2 $ is the golden ratio.

The identity between the multicontinuous fraction and branching infinite radical in \eref{eq:gk2} and \eref{eq:gk3} is very general, can be stated in the formal sense, but require in $\C$ additional assumption to get convergence. Setting
\[\alpha_k = \ell_k {\sf p}_k,~~\beta_k = r_k{\sf p}_{k+1},\] 
\ben\label{eq:gk2a}
g_k&=&\frac{1}{1-\l[s_k+ \alpha_k g_{k} +\beta_k g_{k+1}\r]},\\
g_k&=&\frac { 1- \beta_k \,g_{k+1}\,-s_k-\sqrt { \beta_k^2 {g_{k+1}}^{2}+2\,\beta_k \,g_{k+1}(s_k-1)-4\,\alpha_k +({s_k}-1)^2} }{2\,\alpha_k}
\een
Formula as \eref{eq:sgef}, is in fact a relation between 3 quantities, and then, to go beyond  the identity between \eref{eq:gk2} and \eref{eq:gk3} a third point of view is needed.

\paragraph{When the tree $T$ is not a regular $d$-ary tree,} then the same expansion remains valid but in
\eref{eq:tree_CT}, since $c_T(u)$ depends on $u$, the expansion is not regular.
We see that the algebraic expansion of $\mathcal{W}_{\root}$, takes progressively the same shape as the tree $T$. In particular, if $v$ is a leaf, $\mathcal{W}_v=1$, the tree-like development  of the multicontinuous fractions expansion stops at the corresponding place. In the other end, each end in the tree, produce and infinite ``pile'' of fractions.

\subsubsection{Multicontinuous fractions associated with \AUD chains}

For general \AUD chains, jumps to the descendants are allowed. Let us call $E_u$ the set of excursions (i.e. paths) in $T_u$, starting and ending at $u$. We again assume that any step $(a,b)$ is still weighted by $w_{a,b}$ which is either a complex number (so that some of the following considerations may only hold on a convergence domain) or a formal variable. 

An excursion in $E_u$ can be decomposed:\\
-- It is either a path reduced to the root of $T_u$, i.e. $p=(u)$,\\
-- It may start with a step $(u,u)$ then followed by any path from $E_u$,\\
-- or it may start with a jump  $(u,v)$ with $v\in T_u$, $v\neq u$. Denote by $v_0=v, v_1,\cdots,v_h=u$ be the ancestors of $v$ (from $v$ to $u$). A path in  $E_u$ starting from $(u,v)$ can be decomposed as $(u,v_0),e_0,(v_0,v_1),e_1,\cdots,e_{h-1},v_{h-1},v_{h},e_h$, where $e_i$ is any excursion in $E_{v_i}$. \par
From that we find that $W_u$ the total weight of excursion in $T_u$ is 
\ben W_u &=& 1+\l[w_{u,u}+\sum_{v \in T_u \setminus\{u\}} w_{u,v}\prod_{a\in\cror{u,v}} w_{a,p(a)} W_a\r]W_u\\
             &=& \frac{1}{1-\l[w_{u,u}+\sum_{v \in T_u \setminus\{u\}} w_{u,v}\prod_{a\in\cror{u,v}} w_{a,p(a)} W_a\r]},\een
which is again a multicontinuous fractions formula. Using the same approach as in \eref{eq:dqg} and \eref{eq:qgqd}), it can be used to compute Green functions, the probability of returning to before touching a given level or node, etc.

\subsection{A family of Markov chains on the rational numbers}
Set $\N^{+}=\{1,2,\cdots\}$, and let us define $\Q^+:=\{ a/b, a \in \mathbb{N}^+,b\in \N^+, {\sf gcd}(a,b)=1\}$ the set of rational numbers in the reduced standard form.
For each element $a/b \in\Q^+$, define a notion of left $L(a/b)$ and right $R(a/b)$ child of $a/b$ as 
\[L\left(\frac{a}{b}\right)=\frac{a}{a+b}\quad \text{and }\quad  R\left(\frac{a}{b}\right)=\frac{a+b}{b}.\]
Since $R(a/b)>1$ and $L(a/b)<1$, each element $a/b$ can only be in the image of $L$ or in the image of $R$, but not both. Define the parent of $a/b$ when $a>b$ as $P(a/b)=a'/b'\in \Q^{+}$ such that $R(a'/b')=a/b$ (so that $P(a,b)=(a-b)/b$), and when $a/b<1$, the parent  of $a/b$ is $a'/b'\in\Q^{+}$ such that $L(a'/b')=a/b$, and then $P(a/b)=a/(b-a)$~:
\[P\left(\frac{a}{b}\right) = \1_{a>b}\frac{a-b}{b} + \1_{a<b}\frac{a}{b-a}.\]
Notice that this definition does not define the ancestor of $1/1$: set instead $P(1/1)=1/1$.

It is possible to define a (time-homogeneous) Markov chain $(X_0,X_1,\cdots)$ on $\Q^{+}$ by defining the transition probabilities as follows. For each $a/b\in \Q^+$ take four numbers $(r(a/b), \ell(a/b), p(a/b), s(a/b))\in[0,1]^4$ summing to 1 such that and conditional on $X_i=a/b$:\\
$\bullet$ $X_{i+1}=R(a/b)$ with probability $r(a/b)$,\\
$\bullet$ $X_{i+1}= L(a/b)$ with probability $\ell(a/b)$,\\
$\bullet$ $X_{i+1}=P(a/b)$ with probability $p(a/b)$,\\
$\bullet$ $X_{i+1}=a/b$ with probability $s(a/b)$ (the stay probability).

We call transition family the sequence $[(r(x),\ell(x),p(x),s(x)),x\in \Q^+]$ (and of course, it plays the same role as a transition matrix; it is just presented differently). 
Now, we claim the following (simple but fun) result:
\begin{pro}  Every transition family $[(r(x),\ell(x),p(x),s(x)),x\in \Q^+]$ such that all the $r(x)$, $\ell(x)$, $p(x)$ are positive (except $p(\root)=0$) is irreducible on $\Q^+$.
  \end{pro}
  Irreducibility here entails that 1/1 is accessible from any element of $\Q^+$: this is a child's version of the Syracuse problem!

  This proposition is a consequence of the so-called Stern--Brocot tree (SBT) construction (originally in Stern \cite{MR1579066}, Brocot \cite{Brocot}), which is a binary tree, which is marked $1/1$ at the root, and for each node marked by a rational number $a/b$ the left child is marked $L(a/b)$ and the right one $R(a/b)$. It is known that this tree contains each rational number of $\Q^+$ exactly once.
  If we take the usual binary tree with nodes are words in the alphabet $\{0,1\}$, for example $01011$, and make the correspondence $R\to 1$ and $L\to 0$. For that, we can construct a map $f: \{0,1\}\times \Q^+ \to \Q^+$ to encode this, by setting, for $x\in \Q^+$,
  \[ f(0,x)= L(x),~~f(1,x)=R(x);\]
  the node $u \in T_2$ (the infinite binary tree), the rational associated is $q(\root)=1$ and for $u=(u_1,\cdots,u_h)\neq \root$,
\ben q(u)= f(u_h,f(\cdots,f(u_1, 1/1) \cdots)).\een
   
Thus, any random walk on $T_2$ can be expressed in terms of the SBT, and every recurrent Markov chain on $T_2$ has an equivalent on the SBT that is also recurrent.
For example, if we play $L$ with probability $1/4$, $R$ with probability $1/4$ and $P$ with probability $1/2$, then, starting from any rational number, we will visit 1/1 infinitely often. 

\begin{rem}More generally, any transition matrix $\M$ on the rational numbers corresponding via the SBT correspondence to an \AUD transition matrix $\U$ can be solved: it is possible to compute the invariant measure $\pi$, and to decide if it is positive recurrent (and possibly to decide if it recurrent).
  \end{rem}

\section{Examples}
\label{sec:example}

In this section we present a set of examples that are used throughout the paper. We start by introducing a technique motivated by Section \ref{sec:leaf} and more precisely Theorem \ref{theo:lgp}.
\paragraph{Computation by leaf addition (CLA):}
We start by fixing the value at the root $\rho_{\root}=1$ and in order to discover the value of the $\mathsf{h}$-invariant measure at a point $u$ of the tree, it is sufficient to consider the path $\cro{\root,u}$ and use the leaf addition technique as follows: the path $\cro{\root,u} = v_0,v_1,\dots, v_{|u|}$, where $v_0=\root$ and $v_{|u|}= u$ and we consider the family of growing paths (which are trees) $T_i = \cro{\root,v_i}$. 
Then, by \eqref{eq:fgeht},
\begin{align}\label{eq:lafo}
	\rho_u = {\sum_{v\in \cro{\root,p(u)}} \rho_{v} \U_{v,T_{u}}}/\,{\U_{u,p(u)}}.
\end{align}   
{\bf Notation} For a finite set $E$ denote by $\uniform{E}$ the uniform distribution over $E$

In the examples, we will use the following \textbf{expansion principle}: assume that a function $g:T\to \R$ satisfies $g(\root)=1$ and for $u\neq \root$, the recurrence,
\ben
g(u)&=&\sum_{v \in \cro{\root,p(u)}} g(v) F(v,u)
\een for some function $F:T^2\to \R$.
In this case, by expansion
\ben\label{eq:qfzr} g(u)&=& g(\root)\sum_{\ell=1}^{|u|} \sum_{(v_0,\cdots,v_{\ell})} \prod_{i=1}^{\ell} F(v_{i-1},v_i)\een
where the sequence $(v_0,\cdots,v_{\ell})$ starts at $v_0=\root$, ends at $v_\ell=u$ and is increasing along the line $\cro{\root,u}$ in the sense that $|v_{i-1}|< |v_i|$.
Denote by ${\sf Seq}(u,\ell)$ the set of such sequences.

\begin{exa}\label{ex:2}[Uniform \AUD in a finite tree]\\
  \normalfont Let $\U$ be a \AUD transition matrix on a finite tree $T$ characterized as follows: if  $(X_i,i\geq 0)$ is a $\U$-Markov chain, then knowing  $X_i=u$, then    $X_{i+1}$ is chosen uniformly in $T_u \cup\{p(u)\}$ for $u\neq \root$, and uniformly in $T_\root = T$ if $u=\root$.  
	We show how to retrieve the value of $\rho_u$ from the \textbf{(CLA)}. First we set $\rho_\root=1$ and by \eref{eq:lafo}, for  
 $u\neq \root$ write,	
    \ben
		\rho_u &=& \dis\sum_{v\in \cro{\root,p(u)}} \rho_{v} F(v,u)
    \een
    with \ben F(v,u):=\frac{ |T_u|/(|T_v|+1_{v\neq \root})}{1/(|T_u|+1)}=\frac{|T_u|(|T_u|+1)}{|T_v|+1_{v\neq \root}},\een
    and then, we can use the expansion principle for $g=\rho$ with the function $F$.
Some telescopic simplifications arise using Formula \eref{eq:qfzr}, and we get
 \begin{align}
   \rho_u &= \frac{|T_{u}|+1}{|T_\root|}\sum_{\ell=1}^{|u|} \sum_{(v_0,\cdots,v_{\ell})\in{\sf Seq}(u,\ell)}\prod_{i=1}^{\ell} |T_{v_{i}}|\\
        &=\frac{|T_{u}|+1}{|T_\root|} \l[  \prod_{v \in \croc{\root,u}} \left(1+|T_{v}|\right) \r]|T_{u}| 
 \end{align} this last identity can be easily checked:  the expansion of  $\prod_{v \in \croc{\root,u}} \left(1+|T_{v}|\right)$ gives the sum on all increasing sequences  $s=(u_{m_1},\cdots,u_{m_\ell})$ made of elements of $ \croc{\root,u}$, of any size $\ell$,  of the product of $|T_{u_{m_j}}|$.

\end{exa}

\begin{exa}\label{exa:FC}[The Markov chain: with probability $p$ choose a uniform descendant,otherwise the parent]\\
	\normalfont
	Let $T$ be a finite tree, and $p\in(0,1)$ a parameter. 
	Consider the \AUD  transition matrix $\U$ defined by $\U_{\root,.}=p\uniform{T_\root} +(1-p)\delta_{\root}$, and for $u\neq \root$, $\U_{u,p(u)}=1-p$, $\U_{u,v}=p\uniform{T_u}$.\par
	We compute the invariant measure using  the \textbf{CLA} : we set $\rho_\root=1$ and consider $u\in T\backslash \{\root\}$.  By equation \eref{eq:lafo}, we have $\rho_u  = \sum_{v\in \cro{\root,p(u)}} \rho_{v} F(v,u)$ for 
    $F(v,u) = \left(p/(1-p)\right)|T_{u}|/|T_v|$. By the expansion principle
    \begin{align}
	\rho_u	&=\sum_{\ell=0}^{|u|}\sum_{(v_0,\cdots,v_{\ell})} \left( \frac{p}{1-p} \right)^\ell\prod_{i=1}^{\ell} \frac{|T_{v_{i}}|}{|T_{v_{i-1}}|}\\
            &=\frac{|T_{u}|}{|T_{\root}|}\sum_{\ell=1}^{|u|} \binom{|u|-1}{\ell-1} \left( \frac{p}{1-p} \right)^\ell=\frac{|T_{u}|}{|T_{\root}|}p (1-p)^{-|u|}
 \end{align}
 because the cardinality of $\l|{\sf Seq}(u,\ell)\r|$ is given by the number of ways to choose $\ell-1$ elements in $\rrbracket \root , u \llbracket$. 
\end{exa}

\begin{exa}\label{Ex_line}[Markov chain with a simple ``distance to the root projection'' in the complete $d$-ary tree]\\
	\normalfont
	Let $T$ be an infinite complete $d$-ary tree (each node has exactly $d$ descendants, for some positive integer $d$, so that the number of descendants of any node at distance $k$ is $d^k$).\par
     Take an almost upper triangular transition matrix ${\cal  U}=({\cal U}_{a,b},a,b\geq 0)$ on $\mathbb{N}$ (meaning that ${\cal U}_{i,j}>0$ implies that $j\geq i-1$), that we suppose irreducible . 
	Consider the \AUD  transition matrix $\U$ defined by $\U_{\root,v}= {\cal U}_{0,|v|}/d^{|v|}$ for all $v\in T$, and for $u\neq \root$ and $v\in T$, $\U_{u,p(u)}={\cal U}_{|u|,|u|-1}$, $\U_{u,v}= \1_{v\in T_u} {\cal U}_{|u|,|v|}/d^{|v|-|u|}$.\\
	The following representation can be easily checked from the invariance equations
	\begin{align}
		\rho_u 	&={\rho^{\N}_{|u|}}{d^{-|u|}},
	\end{align}
	where $\rho^{\N}$ denotes the invariant measure of $\mathcal{U}$. The result is intuitive in the sense that $\rho$ corresponds to a symmetric partition of $\rho^\N_h$ along the nodes of height $h$.  
    Notice that $\rho^\N$ is unique and has an explicit expression since $\mathcal{U}$ is an almost upper triangular transition matrix (see \cite{MR4669754}).
	However, it is important to notice that this expression corresponds to the $\sf{h}$- invariant measure of the matrix $\U$ but it does not give any hint on how to obtain the set of all left eigenvectors, which in this case has infinite dimension, i.e. $Q(T)=\infty$. The uniqueness of the left eigenvectors for $\mathcal{U}$ tells us that for the other left eigenvectors $\nu$ for $\U$, there exists some level $h$, such that $\nu_u$ is not constant on the set of nodes $u$ at that level. This is consistent with the content of \Cref{sec:LECTM}, where during the construction of a pre-solution, each branching points (for which at least two subtrees are infinite) provides a space of solutions, only one of which is symmetric. 
\end{exa}
\begin{exa}\label{exa:d-ary}[Jump to the leaves or to the parent]\\
	\normalfont Take an infinite $d$-ary tree, i.e. each node has d children or none, with a countable number of ends and a transition matrix given by
	$\U_{u,p(u)}=(1-p)\1_{u\neq \root}$ and $\U_{u,v}=\1_{v \in \partial T, u\preceq v} d^{|v|-|u|} (p\1_{u\neq \root} + \1_{u=\root})$; i.e. with probability $1-p$ we make a transition from $u$ to the parent (if there is one) and with probability $p$ we go to a leaf of the subtree hanging from $u$ and we make the transition to a particular leaf as $d^{-1}$ power the distance of this leaf from $u$. 
	
	We use the \textbf{CLA} again, setting $\rho_\root=1$, and for  nodes $u\neq \root$, $\rho_u=\sum_{v\in \cro{\root,p(u)}} \rho_{v} F(v,u)$ with
    \[F(v,u)= \left(\frac{p}{1-p}\right)d^{|v|-|u|}\] since $\U_{v,T_{u}} = p d^{|v|-|u|}$.
    By the expansion principle,
    \begin{align}
      \rho_u&		=\sum_{\ell=1}^{|u|}\sum_{(v_0, \cdots,v_{\ell})\in{\sf Sq}(u,\ell)} \left( \frac{p}{1-p} \right)^\ell d^{-|u|}
		=d^{-|u|}p(1-p)^{-|u|}.
	\end{align}
\end{exa}
 
\begin{exa}\label{ex:3}[Tree with two ends, two non colinear invariant measures, ``one end positive recurrent'' and ``one end transient''. ]\\
  \normalfont
  The following example is useful to understand the ``projection onto ends'', and maybe to see that the positive recurrence on one end, does not reduce the dimension of the cone of positive invariant measures. \par
  Consider the transition matrix $\M_{i,i+1}=2/3$, $M_{i,i-1}=1/3$ of the simple biased Markov chain on $\Z$. Now, see $\Z$ as a tree rooted at 0, with two ends, $\p^+=(0,1,2,\cdots,)$ and $\p^-=(0,-1,-2,\cdots)$.
  The two corresponding projected transition matrices are given by
 \[ \begin{array}{cclcclccll}
  \M^{\p^+}_{i,i+1} &=& 2/3, &\M^{\p^+}_{i,i-1}&=& (1/3)1_{i>0}, &\M^{\p^+}_{i,i}&=& (1/3)\1_{i=0}, &\textrm{ for }i \geq 0\\
   \M^{\p^-}_{i,i+1} &=& (2/3)\, 1_{i<0}, &\M^{\p^+}_{i,i-1}&=&1/3, &\M^{\p^+}_{i,i}&=& (2/3)\,\1_{i=0}, &\textrm{ for }i\leq 0.
    \end{array}\]
  The standard theory of Markov chains, or the use of criteria for birth and death processes, allows us to see that $\M^{\p+}$ is transient, $\M^{\p^-}$ is positive recurrent. From \Cref{pro:qff}, we know that the left 1-eigenspace of $\M$ has dimension 2. One of the invariant measures is $\rho=(1, 1\in \Z)$. To find a second positive invariant measure $\rho'=(\rho'_k,k\in \Z)$, fix $\rho'_0=1$, and $\rho'_1=1+x$, and solve the system $\rho'_j=(2/3)\rho_{j-1}'+(1/3)\rho'_{j+1}$, one finds: 
   \[ \rho'_j = 1 - (1-2^{j})x.\] 
For any $x\in(0,1]$, this measure is positive (and linearly independent of $\rho$).
\end{exa}

\begin{exa}\label{sec:qfgrht}[Infinite number of ends: transient Markov chain with infinite cone of invariant measures where all the chains associated with each end are recurrent]\\
	\normalfont
Here we give an example of a $\M$-random walk on the complete binary tree  $T=\cup_{k\geq 0}\{0,1\}^k$ where the set of invariant measures is an infinite dimensional cone. This multiplicity implies that $\M$ is transient, while, each chain associated with each end (as defined in \Cref{sec:CMKE}) is recurrent. 
In the tree $T$ all nodes have degree 3 except the root which has degree 2.

Consider the transition matrix of the simple random walk $\M$ on $T$, defined by $\M_{u,p(u)}=\M_{p(u),u}=1/3$ for all $u\neq \root$, and $\M_{\root,1}=\M_{\root,0}=\M_{\root,\root}=1/3$. 

For each end  $\p$ identified with $\N$, the transition matrix $\M^\p$  (see \Cref{sec:CMKE}) is a birth and death process with transition matrix $K$ with entries $K_{0,0} = 2/3$, $K_{0,1} = 1/3$ and  $K_{i,i-1}= K_{i,i}=K_{i,i+1}=1/3$ for every $i\in \N\setminus \{0\}$. Hence each $\M^\p$ is recurrent. 

Now, we claim that for all $x\in[0,1/2]$, there exists an invariant measure $\rho^{(x)}=\l(\rho_u^{(x)},u\in T\r)$, that is solution to $\rho \M=\rho$, and such that (near the root),
\ben\label{eq:near0} \l(\rho^{(x)}_\root, \rho^{(x)}_0,\rho^{(x)}_{1}\r)=\big(1,1+x,1-x\big).\een
In fact, the $h$-invariant measure of $\U$ is the constant measure $(1, u\in T)$; near the root, it coincides with the case $x=0$, given in \eref{eq:near0}.  

Once the values \eref{eq:near0} are fixed, in order to construct $(\rho^{(x)}_u,u\in T)$ invariant by $\M$, it suffices to fix some $\rho^{(x)}_u$ such that each $\rho^{(x)}_v$ is the mean of its neighboring values $\rho^{(x)}_w$ (for $v\neq \root$). A rapid inspection shows that there are many solutions to this problem! It is less easy to find positive solutions.

Let us choose the same weight for siblings $\rho^{(x)}_{u0}=\rho^{(x)}_{u1}$ for all $u$ at depth at least 2: eventually, all strict descendants $u$ of $0$ such that $|u|=k$ will have weight $\rho_u^{(x)}=a_k$, all descendants $u$ of $1$, such that $|u|=k$  will have weight $\rho_u^{(x)}=b_k$ for some sequence $(a_k)$ and $(b_k)$ that we now determine.\par
We thus have $a_1=1+x$ and  $b_1=1-x$. Let us additionally set  $a_0=b_0=1$, which is consistent with $\rho_\root^{(x)}=1$. 

\paragraph{Control of the sign.} 
In order that $\rho^{(x)}$ satisfies the balance equation at a node $u$ at depth  $|u|=k\geq 1$, in the subtree $T_0$, we need 
$2 a_{k+1}+a_{k-1}=3a_k$; this gives $a_k=1+\frac{2^{k}-1}{2^{k-1}}x$
(they are positive, and even, increasing). While in the right subtree $T_1$, by the same reasoning, we need (for $k\geq 1$),  $2 b_{k+1}+b_{k-1}=3b_k$, which gives $b_k=1-\frac{2^{k}-1}{2^{k-1}}x$
and we see that this is positive for all $x \in [0,1/2]$.

To produce other (linearly independent solutions): it suffices to make a bit more abstract what has been done. For any node $u$, the measure
\[\rho^{(u)}=\big(\rho^{(u)}_v,v\in T\big)= \big(\1_{v \in T_{u0}}-\1_{v\in T_{u1}},v\in T\big)\]
that is $\rho^{(u)}_v=1$ (resp. -1) iff $v$ a strict left (resp. right) descendant of $u$, and zero otherwise. 
Clearly, the measures $(\rho^{(u)}, u \in T)$ are linearly independent. Moreover, all measures of the type
  \[\mu:=\pi + \sum_{u \in T}  \lambda_u \rho^{(u)}\]
  for some constants $(\lambda_u)\in \C^{T}$  are well defined and invariant by $\U$: indeed, the value $\mu_v$ can be expressed as the finite sum $1+ \sum_{u\in\cro{\root,v}} \lambda_u . {\sf Sign}(u,v)$ with a ${\sf Sign}(u,v)$ being $+1$ or $-1$ depending on whether $v$ is in the left or right subtree of $u$ and $0$ if it does not belong to either of these trees. But, in any case, by taking the $(\lambda_v,v\in T)$ real numbers small enough so that along each end $\p$, $\sum_{v\in \p} |\lambda_v|<1$, the (row vector) measure $\mu$ has all its coordinates positive; this means that the cone of invariant positive measures is infinite dimensional.

\end{exa}

\bibliographystyle{abbrv}


\begin{thebibliography}{10}

\bibitem{MR2438700}
E.~Aid\'ekon.
\newblock Transient random walks in random environment on a {G}alton-{W}atson
  tree.
\newblock {\em Probab. Theory Related Fields}, 142(3-4):525--559, 2008.

\bibitem{AB}
M.~Albenque and J.~Bouttier.
\newblock Constellations and multicontinued fractions: application to
  {E}ulerian triangulations.
\newblock volume~AR of {\em Discrete Math. Theor. Comput. Sci. Proc.}, pages
  805--816. Assoc. Discrete Math. Theor. Comput. Sci., Nancy, 2012.

\bibitem{Al90}
D.~J. Aldous.
\newblock The random walk construction of uniform spanning trees and uniform
  labelled trees.
\newblock {\em SIAM Journal on Discrete Mathematics}, 3(4):450--465, 1990.

\bibitem{MR3363457}
R.~Ap\'{e}ry.
\newblock Irrationalit\'{e} de {$\zeta_2$} et {$\zeta_3$}.
\newblock Number~61, pages 11--13. 1979.
\newblock Luminy Conference on Arithmetic.

\bibitem{MR2969494}
G.~Ben~Arous and A.~Hammond.
\newblock Randomly biased walks on subcritical trees.
\newblock {\em Comm. Pure Appl. Math.}, 65(11):1481--1527, 2012.

\bibitem{Bremaud}
P.~Br\'{e}maud.
\newblock {\em Markov chains---{G}ibbs fields, {M}onte {C}arlo simulation and
  queues}, volume~31 of {\em Texts in Applied Mathematics}.
\newblock Springer, Cham, second edition, 2020.

\bibitem{Brocot}
A.~Brocot.
\newblock Calcul des rouages par approximation, nouvelle méthode.
\newblock {\em Revue Chonométrique}, 3(186–194), 1861.

\bibitem{Bro89}
A.~Z. Broder.
\newblock Generating random spanning trees.
\newblock In {\em FOCS, vol. 89}, pages 442--447, 1989.

\bibitem{CF}
P.~Cartier and D.~Foata.
\newblock {\em Problèmes combinatoires de commutation et réarrangements}.
\newblock Springer-Verlag, Lecture notes in mathematics, 1969.

\bibitem{MR2184096}
T.~Duquesne.
\newblock Continuum tree limit for the range of random walks on regular trees.
\newblock {\em Ann. Probab.}, 33(6):2212--2254, 2005.

\bibitem{Euler}
L.~Euler.
\newblock De fractionibus continuis dissertatio.
\newblock {\em Commentarii Academiae Scientiarum Petropolitanae 9}, page
  98–137, (1737).
\newblock English translation in Math. Systems Theory 18, 295–328 (1985).

\bibitem{Flajolet}
P.~Flajolet.
\newblock Combinatorial aspects of continued fractions.
\newblock {\em Discrete Math.}, 32(2):125--161, 1980.

\bibitem{FG}
P.~Flajolet and F.~Guillemin.
\newblock The formal theory of birth-and-death processes, lattice path
  combinatorics and continued fractions.
\newblock {\em Adv. in Appl. Probab.}, 32(3):750--778, 2000.

\bibitem{MR4669754}
L.~Fredes and J.-F. Marckert.
\newblock Almost triangular {M}arkov chains on {$\Bbb {N}$}.
\newblock {\em Electron. J. Probab.}, 28:Paper No. 156, 44, 2023.

\bibitem{LFJFM}
L.~Fredes and J.-F. Marckert.
\newblock A combinatorial proof of {A}ldous-{B}roder theorem for general
  {M}arkov chains.
\newblock {\em Random Structures Algorithms}, 62(2):430--449, 2023.

\bibitem{MR3098688}
A.~Hammond.
\newblock Stable limit laws for randomly biased walks on supercritical trees.
\newblock {\em Ann. Probab.}, 41(3A):1694--1766, 2013.

\bibitem{HLT21}
Y.~Hu, R.~Lyons, and P.~Tang.
\newblock A reverse {A}ldous--{B}roder algorithm.
\newblock In {\em Annales de l'Institut Henri Poincar{\'e}, Probabilit{\'e}s et
  Statistiques}, volume 57 - 2, pages 890--900. Institut Henri Poincar{\'e},
  2021.

\bibitem{MR4260489}
Y.~Hu, R.~Lyons, and P.~Tang.
\newblock A reverse {A}ldous-{B}roder algorithm.
\newblock {\em Ann. Inst. Henri Poincar\'e{} Probab. Stat.}, 57(2):890--900,
  2021.

\bibitem{K1847}
G.~Kirchhoff.
\newblock Ueber die aufl{\"o}sung der gleichungen, auf welche man bei der
  untersuchung der linearen vertheilung galvanischer str{\"o}me gef{\"u}hrt
  wird.
\newblock {\em Annalen der Physik}, 148(12):497--508, 1847.

\bibitem{CK}
C.~Krattenthaler.
\newblock The theory of heaps and the {Ca}rtier–{F}oata monoid, 2006.

\bibitem{LR}
F.~Leighton and R.~Rivest.
\newblock Estimating a probability using finite memory.
\newblock In M.~Karpinski, editor, {\em Foundations of Computation Theory},
  pages 255--269, Berlin, Heidelberg, 1983. Springer Berlin Heidelberg.

\bibitem{LPP}
R.~Lyons, R.~Pemantle, and Y.~Peres.
\newblock {Conceptual Proofs of $L$ Log $L$ Criteria for Mean Behavior of
  Branching Processes}.
\newblock {\em The Annals of Probability}, 23(3):1125 -- 1138, 1995.

\bibitem{MR1410689}
R.~Lyons, R.~Pemantle, and Y.~Peres.
\newblock Biased random walks on {G}alton-{W}atson trees.
\newblock {\em Probab. Theory Related Fields}, 106(2):249--264, 1996.

\bibitem{MR1601753}
R.~Lyons, R.~Pemantle, and Y.~Peres.
\newblock Unsolved problems concerning random walks on trees.
\newblock In {\em Classical and modern branching processes ({M}inneapolis,
  {MN}, 1994)}, volume~84 of {\em IMA Vol. Math. Appl.}, pages 223--237.
  Springer, New York, 1997.

\bibitem{MR2365486}
Y.~Peres and O.~Zeitouni.
\newblock A central limit theorem for biased random walks on {G}alton-{W}atson
  trees.
\newblock {\em Probab. Theory Related Fields}, 140(3-4):595--629, 2008.

\bibitem{MR4602842}
A.~D. Sokal.
\newblock A simple algorithm for expanding a power series as a continued
  fraction.
\newblock {\em Expo. Math.}, 41(2):245--287, 2023.

\bibitem{MR1579066}
M.~Stern.
\newblock Ueber eine zahlentheoretische {F}unktion.
\newblock {\em J. Reine Angew. Math.}, 55:193--220, 1858.

\bibitem{AV}
A.~Varvak.
\newblock {\em Encoding properties of lattice paths}.
\newblock PhD thesis, Brandeis University, 2004.

\bibitem{VX}
G.~X. Viennot.
\newblock Heaps of pieces, i: Basic definitions and combinatorial lemmas.
\newblock In {\em Combinatoire {\'e}num{\'e}rative}, pages 321--350. Springer,
  1986.

\bibitem{DZ}
D.~Zeilberger.
\newblock A combinatorial approach to matrix algebra.
\newblock {\em Discrete Mathematics}, 56(1):61 -- 72, 1985.

\end{thebibliography}

\newpage
\tableofcontents
 \end{document}